\documentclass[12pt,reqno]{amsart}
\usepackage{color}
\usepackage{latexsym}
\usepackage{amsmath,amssymb,dsfont,amsfonts}
\usepackage{mathrsfs}
\usepackage{verbatim}
\usepackage{graphicx}
\usepackage{graphics}
\usepackage{geometry}
\usepackage{booktabs}
\usepackage[shortlabels]{enumitem}
\usepackage{xcolor}
\usepackage{fancyhdr}
\newtheorem{remark}{Remark}[section]
\newtheorem{proposition}[remark]{Proposition}
\newtheorem{theorem}[remark]{Theorem}
\newtheorem{definition}[remark]{Definition}
\newtheorem{corollary}[remark]{Corollary}
\newtheorem{lemma}[remark]{Lemma}
\newtheorem{assumption}[remark]{Assumption}

\def\R{\mathbb R}
\def\N{\mathbb N}

\def\E{\mathbb E}
\def\P{\mathbb P}
\def\Q{\mathbb Q}

\def\shc{{\mathcal C}}
\def\shd{{\mathcal D}}
\def\shf{{\mathcal F}}

\def\shl{{\mathcal L}}

\newenvironment{prooff}{{\bf \textit{Proof}}}{\hfill $\Box$ \\}
\newenvironment{preuve}{{\bf \textit{Proof.}}}{\hfill $\Box$ \\}
\DeclareMathOperator*{\sgn}{sign}
\numberwithin{equation}{section}

\allowdisplaybreaks

\title{ON SDE\MakeLowercase{s} FOR BESSEL PROCESSES IN LOW DIMENSION AND PATH-DEPENDENT EXTENSIONS}

\author{Alberto OHASHI$^1$}

\author{Francesco RUSSO$^2$}

\author{Alan TEIXEIRA$^1$}

\address{$1$ Departamento de Matem\'atica, Universidade de Bras\'ilia, 70910-900, Bras\'ilia, Brazil.} \email{amfohashi@gmail.com}

\address{$2$ ENSTA Paris, Institut Polytechnique de Paris,
 Unit\'e de Math\'ematiques appliqu\'ees, 828, boulevard des Mar\'echaux, F-91120 Palaiseau, France}
 \email{francesco.russo@ensta-paris.fr}
\address{$3$ ENSTA Paris, Institut Polytechnique de Paris,
 Unit\'e de Math\'ematiques appliqu\'ees, 828, boulevard des Mar\'echaux, F-91120 Palaiseau, France}
\email{alan.teixeira@ensta-paris.fr}

\date{November 7th 2022}

\begin{document}

\begin{abstract}
  The Bessel process in low dimension ($0 \le \delta \le 1$)
  is not an It\^o process and it is a semimartingale
  only in the cases $\delta = 1$ and $\delta = 0$.
 In this paper
  we first characterize it as the unique solution of
  an SDE with distributional drift or more precisely
  its related martingale problem.
  In a second part, we introduce a suitable notion of
  {\it path-dependent Bessel processes} and we characterize
  them as  solutions of path-dependent SDEs with distributional
  drift.
\end{abstract}

\maketitle





\section{Introduction}

\label{SIntro}

The class of Bessel processes is one of the most important classes of
diffusion processes with values in $\R_+$. It is a family of strong Markov processes parameterized by $\delta \in \mathbb{R}_+$ (called the
{\it dimension}), which has deep connections with the radial behavior
of the Brownian motion, square-root diffusions, conformally invariant processes, etc.
Bessel processes have been largely investigated in the literature,
we refer the reader to e.g \cite{mansuy,zamb,Yor}
(Section 2.3, Chapter 3 and Chapter XI, respectively)
 for an overview on Bessel processes.

Let $x_0 \ge 0$.
We recall that a {\it Bessel process} $X$ (with initial condition $x_0$,
dimension $\delta \ge 0$
and denoted $\text{BES}^\delta(x_0)$)
is defined as
the square root of the so-called  {\it squared Bessel process}
 (with initial condition $s_0 = x^2_0$, dimension $\delta \ge 0$
and denoted $\text{BESQ}^\delta(x^2_0)$),
which is characterized as the pathwise unique solution of the SDE
\begin{equation} \label{ESQBessel}
  dS_t = 2 \sqrt{\vert S_t \vert} dW_t + \delta t, S_0 = x_0^2,
  \end{equation}
where $W$ is a standard Brownian motion.

When $\delta > 1$ it is possible to characterize $X$ as (pathwise
unique non-negative)
solution of
\begin{equation}\label{NMBES}
dX_t = \frac{\delta-1}{2} X^{-1}_tdt + dW_t,
\end{equation}
where $W$ is again a standard Brownian motion,
see for instance Exercise (1.26) of Chapter IX in \cite{Yor}.
From now on the letter $W$ will always denominate such a process.
In particular $X$ is an It\^o process.
For $0 \le \delta \le 1$, the integral $\int_0^t X^{-1}_sds$ does not converge and $\text{BES}^\delta(x_0)$ is a non-semimartingale process, except for $\delta = 1$ and
$\delta = 0$,  see \cite{Yor, Jean},  Chapter XI Section 1 and Section 6.1,
respectively.
If $0 < \delta < 1$, see for instance
  \cite{BertoinBessel} it is known that
\begin{equation}\label{NMBES1}
X_t = x_0+ \frac{\delta-1}{2}~\text{p.v.}~\int_0^t\frac{ds}{X_s}ds + W_t, t\ge 0,\end{equation}
where p.v. stands for principal value   defined as
$$ \int_0^t \varphi(X_s) ds := \int_{\R_+} \varphi(a) L^X_t (a) a^{\delta-1} da,$$
where $L^X$ is the local time of $X$, defined as density occupation measure,
 i.e.
for every bounded Borel function $\varphi: \R_+ \rightarrow \R_+$,
see for instance (10.1.3) of \cite{mansuy}.
More precisely one defines
 $$ \text{p.v.}~\int_0^t\frac{ds}{X_s}ds := \int_{\R_+}
 (L^X(a) - L^X(0)) a^{\delta - 2} da.$$

The drift in decomposition \eqref{NMBES1} is a zero energy additive functional in the language of Markov processes and $\text{BES}^\delta(x_0)$ is a Dirichlet process, i.e. the sum of a local martingale and a zero quadratic variation
process. As a consequence, in the low dimensional regime, \eqref{NMBES} does not correctly represent the paths of $\text{BES}^\delta(x_0)$. Representation \eqref{NMBES1} can be interpreted as the Dirichlet process decomposition of $\text{BES}^\delta(x_0)$. For further details, we refer the reader to the works \cite{zamb,ew,mansuy} and other references therein.

Typical examples of low-dimensional Bessel processes appear in the theory of Schramm-Loewner evolution, see e.g. \cite{lawler}. Two-parameter family of Schramm-Loewner evolution $\textit{SLE}(\kappa,\kappa-4)$ defined in \cite{werner} provides a source of examples of $\text{BES}^\delta$ flows with very singular behavior when $\delta = 1-\frac{4}{\kappa}, \kappa> 4$. In fact, the final right-boundary of $\text{SLE}(\kappa, \kappa - 4)$ processes is described by the excursions of $\text{BES}^\delta(x_0)$. We refer the reader to \cite{dubedat} for more details. We also drive attention to \cite{beliaev2020new} for more recent applications of low-dimensional Bessel processes starting at the origin.

In this work, we characterize $\text{BES}^\delta(x_0)$, for $0\le \delta \le 1$, as the unique solution of an SDE with distributional drift. The main result of this paper states that one natural way to investigate the SDE dynamics of low-dimensional Bessel processes is the interpretation of the singular drift $\frac{1}{x}$ as the derivative in the sense of Schwartz distributions of the function $x \mapsto {\rm log} \vert x \vert$ rather than principal values via local times. For this purpose, we interpret \eqref{NMBES} as a strong-martingale problem previously introduced by \cite{russo_trutnau07}. In this case, for $0\le \delta \le 1$, we prove $\text{BES}^\delta(x_0)$ is the unique non-negative solution of a suitable strong-martingale problem starting at $x_0\ge 0$. A non-Markovian extension is also considered for SDEs with singular drifts of the form
$$ \frac{\delta-1}{2}\frac{1}{X_t} + \Gamma(t,X^t),$$
where $\Gamma$ is a path-dependent non-anticipative functional satisfying some technical conditions and $X^t$ will be given in \eqref{etat}.
 Our analysis is inspired by the series of works \cite{frw1, frw2, russo_trutnau07} which treat Markovian SDEs of the form
\begin{equation}\label{SDEone}
dX_t = \sigma(X_t)dW_t + b'(X_t)dt,\ \ X_0 \stackrel{d}{=}
\delta_{x_0},
\end{equation}
where $\sigma$ and $b$ are continuous functions on $\mathbb{R}$.
 Moreover $\sigma$ is strictly positive and one supposes
  the existence of the function
\begin{equation}\label{Sigmaintr}
\Sigma (x) :=  2\int_{0}^{x}\frac{b'}{\sigma^{2}}(y)dy, x \in \mathbb{R},
\end{equation}
 as a suitable limit via regularization.
 We stress that $b'$ is the derivative of some function $b$ in the sense of distributions.
Assuming
\eqref{Sigmaintr},
the Markovian operator $L$ is defined by the authors 
as
\begin{equation}
 \label{DefLbis}
 Lf = (e^\Sigma f')'\frac{e^{-\Sigma}\sigma^2}{2},
\end{equation}
where $ f $ belongs to the domain
 $$ D_L =\{ f \in C^1(\R) \vert  f' e^\Sigma  \in C^1(\R),$$
 see e.g. \cite{frw1}, Section 2 and also
 \cite{ORT1_PartI} Proposition 4.1.

When $\sigma$ and $b'$ are functions
then previous expression equals
\begin{equation} \label{DefL}
Lf = \frac{\sigma^{2}}{2}f'' + b'f'.
\end{equation}



In \cite{ORT1_PartI},  we have studied
the class of SDEs
\begin{equation} \label{X}
dX_t = \sigma(X_t)dW_t + b'(X_t)dt + \Gamma(t, X^t)dt,\ \ X_0 \stackrel{d}{=}
\delta_{x_0},
\end{equation}
 for some classes of functionals $\Gamma$.

In this paper, we will investigate existence and uniqueness
of an SDE of the type \eqref{X}, where
$\sigma = 1$, but $b$ is no more a continuous function.
More precisely, we focus on
the SDE
\begin{equation}\label{sdeint}
dX_ t  = dW_t + b'(X_t)dt + \Gamma(t,X^t)dt,~X_0 \stackrel{d}{=} \delta_{x_0},
\end{equation}
where $b$  is given by
\begin{equation} \label{EbBessel}
b(x) =
\left \{
  \begin{array}{ccc}
\frac{\delta - 1}{2}\log|x|, x \in \R^*  &\vert& \ \delta \neq 1 \\
     H(x),   x \in \R
&\vert& \delta = 1,
\end{array}
\right .
\end{equation}
and   $H$ is the Heaviside function and $\mathbb{R}^* = \mathbb{R}-\{0\}$.
Then, \eqref{NMBES1} is considered as a particular case of the SDE \eqref{sdeint} with distributional drift $b'$ and $\Gamma = 0$.
Even though $b$ is no longer a continuous function,
\eqref{Sigmaintr} can still be defined in such a way that
 $\Sigma  \equiv 2 b$ and (\ref{DefLbis}) holds. We distinguish the two cases: $0 \le \delta < 1$ and $\delta =1$.
\begin{itemize}
\item{$0 \le \delta < 1$.}
  If $b$ is given by \eqref{EbBessel}, then
\eqref{Sigmaintr} implies
\begin{equation} \label{hBesselIntro}
\exp(-\Sigma(x)) = |x|^{1 - \delta}.
\end{equation}
At this point, representation \eqref{DefLbis}
for $L^\delta = L$ yields
   \begin{equation}\label{LDbisdelta}
     L^{\delta}f(x) = \frac{f''(x)}{2} + \dfrac{(\delta -1)f'(x)}{2x}, \
     x \neq 0.
	\end{equation}
\item{$\delta = 1$.}
  In this case, $ b(x) = H(x)$.
So  \eqref{DefLbis} yields
   \begin{equation}\label{LDbisdelta1}
     L^{1}f(x) = \frac{f''(x)}{2} + \delta_0 f'(x), \
     x \neq 0,
	\end{equation}
        where  $\delta_0$ is the Dirac measure at zero.
          Those expressions are perfectly well-defined for $f \in \shd_{L^\delta}$ defined in
       Section \ref{Sdeltasmaller1} below.
        The product $\delta_0 f'$ for $f \in D_L$ is necessarily zero.
      \end{itemize}


We then study the (possibly non-Markovian) martingale problem associated with the operator
$$\mathcal{L}^\delta f = L^\delta + \Gamma f',$$
in a suitable domain. The notion of martingale problem related to $\mathcal{L}^\delta$ is given by Definition \ref{D31}. The notion of strong martingale problem related to the domain of $L^\delta$ and an underlying Brownian motion $W$ is given by Definition \ref{DSolution}, which borrows the one in \cite{russo_trutnau07}. It has to be compared with the notion of strong existence and pathwise uniqueness of an SDE. In particular, it represents the corresponding notion of strong solution of SDEs in the framework of martingale problems.

Sections \ref{Sdeltasmaller1}, \ref{S53}, \ref{S54}
present a series of results concerning existence/uniqueness for the SDE \eqref{sdeint} in Markovian case, for $0 \le \delta < 1$.
In particular, Propositions \ref{S} and \ref{su} show the low-dimensional Bessel process $\text{BES}^\delta(x_0)$ as the unique non-negative solution of the strong martingale problem associated with $L^\delta$ for $0 < \delta < 1$ and $x_0\ge 0$.
A similar discussion concerns the case $\delta = 1$, see Section
\ref{S55}, Propositions \ref{Pdelta=1} and \ref{S1}.
We remark that in the case $\delta = 1$, results for pathwise
uniqueness, see \cite{harri} were already available in the literature.

  In Section \ref{ExtDomain} we connect the martingale problem
  related to Bessel processes to one related to an extended domain
  which includes the harmonic function
\begin{equation} \label{Eharmonic}
  h(x) = {\rm sgn}(x) \frac{{\vert x \vert}^{2 - \delta}}{2-\delta},
  x \in \mathbb R.
\end{equation}
  We also give general conditions
  on the marginal law of a generic process
  which is solution of the basic martingale
  problem to solve the one with extended domain.
  This is fulfilled for instance by the Bessel
  process with dimension $\delta >0$.
  Related considerations are discussed when
  $\delta = 0$.
 
  In Section \ref{SBesselPath}, we establish existence and uniqueness of the  martingale problem associated with the non-Markovian SDE \eqref{sdeint} under the condition that $\Gamma$ is bounded; see Propositions \ref{P61}
 and \ref{P64}. Proposition \ref{exi} proves existence when $ \Gamma$ is 
unbounded 
with some technical conditions.
 Theorem \ref{uni2} illustrate sufficient conditions on $\Gamma$ to have well-posedness of the strong martingale problem.

We highlight that  \cite{pilipenko} has established uniqueness \eqref{NMBES1} of non-negative solutions
$X$, when  $0 \le \delta \le 1$,
under the condition that the solution $X$ spends zero time at the point zero, i.e.,
\begin{equation}\label{ztime}
\mathbb{E} \Big[\int_0^\infty \mathds{1}_{\{0\}}(X_s)ds\Big]=0.
\end{equation}
In contrast to  \cite{pilipenko} we do not suppose that assumption and we provide uniqueness
among all non-negative solutions.

One important objective of the paper is the definition of
{\it path-dependent Bessel process}.
 Let $\delta  \ge 2$ be an integer.
Similarly to the classical Markovian case with integer dimension, 
a path-dependent Bessel type process (as solutuion of \eqref{sdeint})   appears 
considering the dynamics of a $\delta$-dimensional Brownian motion $\beta$ with
 drift having a radial intensity proportional to
a non-anticipative functional $\Gamma$.

%
 More precisely, let $Y$ be a solution to
\begin{equation} \label{EY10}
 dY_t = d\beta_t + \Gamma(t,\Vert Y_s \Vert_{\R^{\delta}}, s \le t) \frac{Y_t}{\Vert Y_t \Vert_{\R^{\delta}}}  dt,
\end{equation}
Then $X_t := \Vert Y_t \Vert_{\R^{\delta}}, $ i.e.
 the Euclidean norm  in $\R^\delta$,
 solves \eqref{sdeint}.
Indeed, if $Y$ is a solution of \eqref{EY10}, then
a formal application of It\^o's formula to $\rho_t: = \Vert Y_t \Vert^2_{\R^\delta}$
and L\'evy's characterization theorem for local martingales
show that
\begin{equation} \label{NBES11}
  d\rho_t = 2\sqrt{\rho_t} dW_t + 2 \sqrt{\rho_t}
  \Gamma (t, \sqrt{\rho_s}, s \le t) dt + \delta dt.
  \end{equation}
A subsequent formal application of It\^o's formula shows that $X_t = \sqrt{\rho_t}$
solves \eqref{sdeint}.
Our result concerns the extension of that model to the singular case represented by $\delta \in [0,1]$.

The paper is organized as follows. After this Introduction
we recall the notations and some important results from \cite{ORT1_PartI}.
Then we introduce specific preliminary considerations. Section \ref{SBessel}
is devoted to the case of Bessel processes in low
dimension, under the perspective of strong martingale problems.
Section \ref{SBesselPath} discusses the case of non-Markovian perturbations
of  Bessel processes.

\section{About path-dependent martingale problems}

\label{SPathDep}

\subsection{Preliminary notations, definitions and results}

$ $

In this section we recall the general notation and some necessary results from \cite{ORT1_PartI}.

Let $I$ be an interval of $\mathbb{R}$.
For $k \in \N$, $C^{k}(I)$ will denote
the space of real functions defined on $I$ having
continuous derivatives till order $k$.
Such space is endowed with the uniform convergence topology on compact sets
for the functions and all derivatives.
Generally $ I = \mathbb{R}$, $\R_+:= [0,+\infty[$, $\R_{-}:= ]-\infty, 0]$,
$[0, T],$ for some fixed
positive real $T$.
If there is no ambiguity $C^{k}(\mathbb{R})$ will be simply indicated by
$C^{k}.$ The space of continuous functions on $I$ will be denoted by $C(I)$.
Given an a.e. bounded real function $f$, $\vert f \vert_\infty$
will denote the essential supremum.

We recall some notions from \cite{frw1}.
For us all filtrations $\mathfrak{F}$ fulfill the usual conditions.
When no filtration is specified, we mean the canonical filtration of an
underlying process.
Otherwise, the canonical filtration associated with a process $X$
is denoted by $\mathfrak{F}^X$.

A sequence $(X^n)$ of continuous processes indexed by $[0,T]$ is said to converge u.c.p.
to some process $X$ whenever $\displaystyle\sup_{t \in [0,T]} \vert X^n_t - X_t \vert$ converges to zero in probability.

We consider a locally bounded
functional
\begin{equation} \label{EGamma}
\Gamma: \Lambda \rightarrow \mathbb{R},
\end{equation}
where
\begin{equation*}
\Lambda:= \{(s, \eta^s), s \in [0, T], \eta \in C([0, T])\}
\end{equation*}
and
\begin{equation} \label{etat}
\eta^{t}_s =
\left\{\begin{array}{ccc}
\eta_s,  &\text{if}& \ s \leq t\\
\eta_t, & \text{if}& \ s >t.
\end{array}
\right.
\end{equation}

By convention, we extend $\Gamma$ from $\Lambda$ to $[0,T] \times C([0,T])$
by setting (in a non-anticipating way)
$$ \Gamma(t,\eta):= \Gamma(t,\eta^t), t \in [0,T], \eta \in C([0,T]).$$
All along the paper $E$ will denote $\R$ or $\R_+$.

Let us consider some locally bounded
Borel functions  $\sigma, b': E \rightarrow \R$.
In this case the path-dependent SDE
\begin{equation} \label{E1.3}
\left \{
\begin{array}{cll}
dX_t  &=& \sigma(X_t)  dW_t + b'(X_t) dt + \Gamma(t,X^t) dt   \\
X_0  & = & \xi,
\end{array}
\right.
\end{equation}
for some deterministic initial condition $\xi$ taking values in $E$,
makes  perfectly sense, see Section 5 of \cite{ORT1_PartI},
in particular one can speak about strong existence, pathwise uniqueness,
existence and uniqueness in law.
\eqref{E1.3} is denominated by $E(\sigma,b',\Gamma)$.
 Proposition 3.2 in \cite{ORT1_PartI} implies the following.

\begin{proposition} \label{P31}
 Let $b': E \rightarrow \R$ be a locally bounded function.
  We set $L f = \frac{\sigma^2}{2} f'' + b' f', \ f \in C^2(E)$.
	A couple $(X,\P)$ is a solution of
	$E(\sigma, b',\Gamma)$,
	if and only if, under $\P$,
	\begin{equation}\label{Ef32}
          f(X_t) - f(X_0) - \int_0^t L f(X_s) ds -
          \int_0^t f'(X_s) \Gamma(s,X^s)ds
	\end{equation}
	is a local martingale, where
	$ Lf = \frac{\sigma^2}{2} f'' + b'f'$,
	for every $f \in C^2(E)$.
      \end{proposition}
In this paper, we will be interested in a formal
$E(\sigma,b', \Gamma)$ where $\sigma = 1$ but $b'$
is the derivative of some specific Borel discontinuous function.
The formulation  is inspired by Proposition \ref{P31}
which states that the SDE is equivalent to a specific martingale problem.
We will consider formal PDE operators of the
type $L: \shd_L(E) \subset C^1(E) \rightarrow C(E)$, where
$L f$ gives formally $\frac{\sigma^2}{2} f'' + b' f'$.
When $b', \sigma$ are locally bounded functions
then $\shd_L(E) = C^2(E)$.
In that case, the notion of martingale problem is (since the works of Stroock and Varadhan \cite{sv})
               is a concept related to solutions of SDEs in law.


\begin{definition} \label{D31}
	\begin{enumerate}	
		\item We say that a continuous stochastic process $X$ solves
		(with respect to a probability $\P$ on some measurable space
		$(\Omega,\shf)$)
		the
		martingale problem related to
		\begin{equation}\label{DOperatorL}
		\shl f := L f + \Gamma f',
		\end{equation}
		with initial condition $\nu = \delta_{x_0}, x_0 \in E$,
		with respect to a domain $\shd_L(E)$
		if
		\begin{equation}\label{lmp}
		M^f_t := f(X_t) - f(x_0) - \int_{0}^{t}\mathit{L}f(X_s)ds - \int_{0}^{t}f'(X_s)\Gamma(s, X^s)ds,
		\end{equation}	
		is a $\P$-local martingale for all $f \in \mathcal{D}_{\mathit{L}}(E)$.

		We will also say that the couple $(X,\P)$ is a solution of (or $(X,\P)$ solves) the
		martingale problem with respect to $\shd_L(E)$.
		\item If a solution exists we say that the martingale problem above
		{\it admits existence}.
		\item We  say that the martingale problem above
		{\it admits uniqueness} if any two solutions $(X^i,\P^i), i  = 1,2$
		(on some measurable space $(\Omega,\shf)$)
		have the same law.
	
	\end{enumerate}
	
\end{definition}
 In the sequel, when the measurable space $(\Omega,\shf)$
		is self-explanatory it will be often omitted.

Below we introduce the analogous notion of strong existence
and pathwise uniqueness for our martingale problem,
see also \cite{ORT1_PartI} for the case when $b'$
is the derivative of a continuous function and
\cite{russo_trutnau07} for the case $\Gamma = 0$.
In both cases we had $E= \R$.
\begin{definition} \label{DSolution}
	$ $
	\begin{enumerate}
		\item Let $(\Omega, \mathcal{F}, \P)$ be a probability space and let
		$ \mathfrak{F} = (\shf_t)$ be the canonical filtration associated with a
		fixed Brownian motion $W$.
		Let $x_0 \in E$.
		We say that a continuous $\mathfrak{F}$-adapted $E$-valued process $X$ such that
		$X_0 = x_0$ is a {\bf solution to
			the strong martingale problem}
		(related to \eqref{DOperatorL}, $\sigma$)
		with respect to $\shd_L(E)$ and $W$ (with related filtered probability space),
		if
		\begin{equation}\label{lmpBis}
                  f(X_t) - f(x_0) - \int_{0}^{t}\mathit{L}f(X_s)ds - \int_{0}^{t}f'(X_s)\Gamma(s, X^s)ds =
                  \int_{0}^{t}f'(X_s) \sigma(X_s) dW_s,
		\end{equation}	
		for all $f \in \shd_L(E)$.
		\item We say that the martingale problem
		related to \eqref{DOperatorL} and $\sigma$ 
		with respect to $\shd_L(E)$
		admits \textbf{strong existence}
		if for every $x_0 \in E$,
		given a filtered probability space
		$(\Omega, \mathcal{F}, \P, \mathfrak{F}),$
		where  $\mathfrak{F} = (\shf_t)$ is the canonical
                filtration associated with
		a Brownian motion $W$,
		there is a  process $X$ solving the strong martingale problem
		(related to \eqref{DOperatorL} and $\sigma$)
		with respect to $\shd_L(E)$
		and $W$
		with $X_0 = x_0$.
		\item We say that the martingale problem
		(related to \eqref{DOperatorL})
		with respect to $\shd_L(E)$
		admits \textbf{pathwise uniqueness}
		if given
		$  (\Omega, \mathcal{F}, \P)$
		and a Brownian motion $W$
		and $X^i, i = 1, 2$ are solutions to the strong martingale problem with respect to
		$\shd_L(E)$ and $W$
		with $\P[X_0^1 = X_0^2] = 1$
		then $X^1$ and $X^2$ are indistinguishable.
	\end{enumerate}
	
\end{definition}

The mention $E$ will be often omitted when $E = \R$.
For instance $C^1(E), C^2(E), \shd_L(E)$, will be simply denoted by
$C^1, C^2, \shd_L$.

\section{Martingale problem for Bessel processes}

\label{SBessel}
$ $

\subsection{Preliminary considerations}

\label{SBessPrelim}

$ $

In this section, we are going to
introduce and investigate well-posedness for
 a martingale problem related to a Bessel process.
In this section again $W$ will denote a standard Brownian motion.
 We recall that the rigorous definition of the Bessel process is the
following. A non-negative process $X$ is said to be
a Bessel process starting at $x_0$ with dimension
$\delta \ge 0$ (notation  $BES^{\delta}\left(x_0\right)$)
if $S = X^2$ is a squared Bessel process starting
at $s_0=x_0^2$ of dimension $\delta$. $S$ is
denoted by
$BESQ^{\delta}\left(s_0\right),$ we recall in particular
that it is the pathwise unique solution of \eqref{ESQBessel}.


As is shown in Proposition 2.13 in Chapter 5 of \cite{ks}
(see also \cite[Chapter 3]{zamb}) \eqref{ESQBessel} admits pathwise uniqueness.
Since $x \mapsto \sqrt{\vert x\vert}$ has linear growth it has weak
existence and so
by
Yamada-Watanabe theorem
it also admits strong existence.

\begin{remark}
 For $\delta > 1,$ we know that the Bessel process $X$
 fulfills
\begin{equation}\label{B}
X_t = x_0 + \frac{\delta - 1}{2}\int_{0}^{t} X_s^{-1}ds + W_t.
\end{equation}
We recall that for $\delta > 2$,
$X$ is even transient and it never touches zero, see \cite[Chapter XI]{Yor}.
As anticipated, when $\delta = 1$ or $\delta = 0$
$X$ is still a semimartingale. Unfortunately
if $0 < \delta < 1$ it is not the case, see Chapter 10 of \cite{mansuy},
it is just a Dirichlet process, i.e. the sum of a local martingale and a
zero quadratic variation process.
\end{remark}
Our point of view consists  in rewriting \eqref{B} under the form
\begin{equation}\label{Blog}
X_t = x_0 + \int_{0}^{t} b'(X_s)  ds + W_t,
\end{equation}
where
$W$ is a Brownian motion and $b'$ is the derivative of the function
 $b(x) = \frac{\delta - 1}{2}  \log \vert x \vert$,
 at least when $\delta < 1$. In other words we make use of
 the ''analytical'' p.v. of  $x \mapsto \frac{1}{ x}$ 
 which is the derivative of $\log.$
 That object is, on $\R$, a Schwartz distribution and not a function, which nevertheless coincides
 with $x \mapsto \frac{1}{ x}$ on $\R^*$.
 This indeed explains \eqref{B} and takes into account
 the ``relevant'' time spent by the Bessel process at zero.


 In the case $\delta = 1$, for similar reasons, and taking into account the fact that
 the Bessel process is a reflected Brownian motion, we naturally choose $b$ to be 
a Heaviside function so that $b'$ is the
$\delta$-Dirac measure at zero.

 We are going to construct two settings: one for $0 \leq \delta <1$ and another one for $\delta =1$.
In what follows, we should recall $\mathbb{R}^* = \mathbb{R}-\{0\}.$

\subsection{The framework for $0 \leq \delta < 1$}
$ $

\label{Sdeltasmaller1}

According to the considerations in Section \ref{SBessPrelim},
the natural form of the operator $L^\delta:= L$ (outside zero)
is expected to be of the form
            \begin{equation}\label{LDter}
              L^\delta f(x) = \frac{f''(x)}{2} + \frac{(\delta -1)f'(x)}{2x},
      \end{equation}
for $f \in C^2(\R^*)$.

      As anticipated, we fix $b:\mathbb{R} \rightarrow \mathbb{R}$,
        $b(x) =\frac{\delta - 1}{2}\log|x|, \ x \neq 0$ and $\sigma\equiv 1$.
        $x \mapsto \frac{\delta - 1}{2 x}, $ appearing in \eqref{LDter}
        coincides
  with $b'$ restricted to
  $\mathbb R^*$.
Formally speaking, $\Sigma$ as in \eqref{Sigmaintr} gives
  $\Sigma(x) = 2 b(x)$, so
\begin{equation} \label{hBessel}
\exp(-\Sigma(x)) = |x|^{1 - \delta}, x \in \R.
\end{equation}

The expression \eqref{LDter} can also be expressed as
\begin{equation} \label{DefLForm}
 L^{\delta} f(x) = \frac{\vert x \vert^{1-\delta}}{2} (\vert x \vert^{\delta-1} f')', x \neq 0.
\end{equation}
The problem is to provide a natural extension for $x = 0,$
which constitutes the critical point.

We have now to specify the natural domain 
of $L^\delta$,
 which is compatible
 with \eqref{DefLForm}.
 \begin{definition} \label{Domain_L_delta}
We will denote by $\shd_{L^\delta}$ the
set of $f \in C(\R) \cap C^2(\R_+) \cap C^2(\R_-)$ such that the following holds.
\begin{enumerate}[label=(\alph*)]
\item
 There is a continuous function $g: \R \rightarrow \R$  extending
 $x \mapsto f'(x) \vert x \vert^{\delta-1}, x \neq 0$.
\item
There is a continuous function $G: \R  \rightarrow \R$,
extending $x \mapsto g'(x) |x|^{1 - \delta}, x \neq 0,$  (i.e.
$ 2 L^\delta f(x), $ according to \eqref{DefLForm})
to $\R$.
\end{enumerate}
We define then
\begin{equation}\label{EG}
  L^{\delta}f:= \frac{G}{2}.
\end{equation}
\end{definition}
\begin{proposition}\label{domain}
  \begin{enumerate}
    \item  Suppose $\delta >0$.
      Then $  \shd_{L^{\delta}} = \shd_\delta := \{f \in C^2(\R)
      \vert f'(0)=0 \}$ and
      \begin{equation}\label{LDbister}
             L^{\delta}f(x) =
 \left\{\begin{array}{ccc}
	\frac{f''(x)}{2} + \frac{(\delta -1)f'(x)}{2x} &:& x \neq 0\\
\delta \frac{f''(0)}{2} &:&  x = 0.
	\end{array}
\right.
\end{equation}
       \item Suppose $\delta =0$.
    Then $  \shd_{L^{0}} = \shd_0,  $
where
\begin{equation} \label{LDter1}
  \shd_0:= \{ f \in
  C^1({\R})  \cap C^2(\R_+) \cap C^2(\R_-) \vert f'(0) = 0
  \}
\end{equation}
\end{enumerate}
 and
\begin{equation}\label{LDbisquater}
             L^{0}f(x) =
 \left\{\begin{array}{ccc}
	\frac{f''(x)}{2} - \frac{f'(x)}{2x} &:& x \neq 0\\
0 &:&  x = 0.
	\end{array}
\right.
\end{equation}
   \end{proposition}
   \begin{proof}

  We first show the inclusion $ \shd_{L^\delta} \subset \shd_\delta $.
Suppose $f \in \shd_{L^\delta}$.
We have
\begin{equation} \label{Ef'}
\lim_{x \rightarrow 0} f'(x) = \lim_{x \rightarrow 0} \vert x
\vert^{1-\delta} g(x) = 0.
\end{equation}
This obviously implies that $f \in C^1(\R)$ and $f'(0) = 0$.
Taking into account \eqref{LDter}, we have
\begin{eqnarray*}
 f''(0+)&:=& \lim_{x \rightarrow 0+} f''(x) =  \lim_{x \rightarrow 0+}
             \left(G(x) - (\delta - 1) \frac{f'(x)}{x}\right) \\
  &=&
 G(0) - (\delta - 1) \lim_{x \rightarrow 0+}  \frac{f'(x)}{x} =
G(0) -  (\delta - 1)   f''(0+),\\
f''(0-)&:=& \lim_{x \rightarrow 0-} f''(x) =  \lim_{x \rightarrow 0-}
 \left(G(x) - (\delta - 1) \frac{f'(x)}{x}\right) \\ &=&
 G(0) - (\delta - 1) \lim_{x \rightarrow 0-}  \frac{f'(x)}{x} =
G(0) -  (\delta - 1)   f''(0-),
\end{eqnarray*}
by L'Hospital rule.
This implies that
\begin{equation} \label{EFii}
\delta f''(0+) = G(0) = \delta  f''(0-).
\end{equation}
To show that  $f \in \shd_\delta$, it remains to show that
$f''(0+) = f''(0-)$ when $\delta \neq 0$.
This obviously follows from \eqref{EFii}, which shows      the inclusion $ \shd_{L^\delta}    \subset  \shd_\delta$ for all $\delta \in [0,1[$.
Now, \eqref{EFii}, \eqref{LDter} and \eqref{EG} show in particular \eqref{LDbister} and \eqref{LDbisquater}.

    We prove now the opposite inclusion
$ \shd_\delta    \subset  \shd_{L^\delta}.  $
     Let $f \in \shd_\delta$, in particular such  that $f'(0) = 0$.
We need to prove that it fulfills the properties (a) and (b)
characterizing $\shd_{L^\delta}$.
 We set $g(x) := f'(x)|x|^{\delta - 1}, x \neq 0$ and $g(0) := 0$.
    By l'Hospital rule we can show that
    $\displaystyle\lim_{x \to 0} g(x) = 0,$ so that $g$ is continuous at zero.
 This proves property (a) characterizing $\shd_{L^\delta}$.
    Taking the derivative of $g$ on $\R^*$ we get
        \begin{equation}\label{g'}g'(x) = f''(x)|x|^{\delta - 1} + (\delta - 1)f'(x) \sgn(x)|x|^{\delta -2}.
        \end{equation}
       Concerning property (b),
as $x \mapsto G(x):=g'(x)|x|^{1 - \delta}$ is continuous on $\R^*$
it is enough to show
that $\displaystyle\lim_{x \to 0} G(x)$ exists.
By \eqref{g'} we obtain $$G(x) = f''(x) + (\delta - 1)\sgn(x)\frac{1}{|x|}f'(x) = f''(x) +
(\delta - 1)\frac{f'(x)}{x}, x \neq 0.$$
We recall that $f''(0+)$ and  $f''(0-)$ exist.
 Taking the limit
 when $x$ goes to zero from the right and from the left,  by L'Hospital rule,
 we get
 \begin{eqnarray*}
   G(0+) &=&  f''(0+) + (\delta - 1) f''(0+) = \delta f''(0+), \\
   G(0-) &=&  f''(0-) + (\delta - 1) f''(0-) = \delta f''(0-).
 \end{eqnarray*}
 Distinguishing the cases $\delta > 0$ (in this case
 $f''(0+) = f''(0-)$) and  $\delta = 0$,
we show that $G(0+) = G(0-)$ and finally
 $G$ extends continuously to $0$.
This concludes the proof of the two properties (a) and (b)
and so the inclusion $ \shd_\delta \subset \shd_{L^\delta}$.
\end{proof}
\begin{remark} \label{Rharmonic}
  In fact one could consider a larger domain
  $\hat \shd_{L^\delta}$ constituted by the functions
  $f \in C(\R) \cap C^2(\R^*)$ fulfilling the conditions
  (a) and (b) before \eqref{EG}.
  Consider for instance the $L^\delta$-harmonic function
  $h$ defined in \eqref{Eh}.
  That function does not belong to $\shd_{L^\delta}$
  because it has no second left and right-derivative in $0$,
but it is an element of $\hat \shd_{L^\delta}$.
 In fact that domain is too large for our purposes
of investigation of well-posedness.
Formulating the martingale problem
replacing $ \shd_{L^\delta}$ with $\hat \shd_{L^\delta}$,
it will be easier to show uniqueness, but more difficult
to formulate existence.
Suppose that $(X,\P)$ is a solution to previous martingale
problem, making use of the domain  $ \shd_{L^\delta}$.
The natural question is to know if $(X,\P)$ is still a
solution to the martingale problem formulated
making use of $ \hat \shd_{L^\delta}$ instead of $\shd_{L^\delta}$.
This will be possible under a restricting condition on the law of $X$,
see Section \ref{ExtDomain}; this condition will be fulfilled
by the Bessel process starting from a point $x_0 \neq 0$ for instance.
\end{remark}

In the sequel we will denote by
$ \shd_{L^\delta}(\R_+)$ the set of functions $f:\R_+ \rightarrow \R$
which are restrictions of functions $\hat f$ belonging to $ \shd_{L^\delta}$.
We recall that, sometimes, we will also denote $ \shd_{L^\delta}(\R):=  \shd_{L^\delta}$.
We will also denote $L^\delta f$ as the restriction to
$\R_+$ of  $L^\delta \hat f$. \eqref{LDbister} shows that
this notation is coherent.
This convention will be made also for $\delta = 1$ in Section \ref{S55}.

Starting from Section \ref{S53}, we will make use
of convergence properties for functions and processes
according to the remark below.

\begin{remark}\label{r1}
	$ $
\begin{enumerate}
\item	If $f:\mathbb{R} \rightarrow \mathbb{R}$ is  continuous
(therefore uniformly continuous on compacts) then $f_{n}(x) = f\left(x + \frac{1}{n}\right)$ converges to $f$ uniformly on compacts.
\item	Let $(\Omega, \shf, \P)$ be a probability space and $X$ a continuous stochastic process on $(\Omega, \shf, \P)$. If $f_n: \mathbb{R} \rightarrow \mathbb{R}$ is a sequence of functions that converges uniformly
 on compacts of $\mathbb R$ to a function $f$ then $f_n(X)$ converges to $f(X)$ u.c.p.
\end{enumerate}
\end{remark}

\subsection{The martingale problem in the full line case when $0 \le \delta <1$.}

\label{S53}

$ $


\begin{proposition}\label{S}
  Let $(\Omega, \shf,\P)$ be a probability space
 and
a
Brownian motion $W$.
Let $x_0 \ge 0, 0 \leq \delta < 1$.
Let $S$  be the solution of \eqref{ESQBessel}
(necessarily non-negative by comparison theorem) 	
with $s_0 = x_0^2,$ so that $X = \sqrt{S}$ is a $BES^{\delta}(x_0)$ process.

Then $X$
 solves the strong martingale problem with respect to $\shd_{L^{\delta}}$
and $W$.
In particular, for every   $f \in \shd_{L^{\delta}}$
\begin{equation}\label{lmp2}
 f(X_t) - f(X_0) - \int_{0}^{t}L^{\delta}f(X_s)ds =
\int_{0}^{t}f'(X_s) dW_s.
 \end{equation}
\end{proposition}
\begin{remark} \label{RS}
\begin{enumerate}
\item
Suppose that $S$ is a non-negative solution of an SDE of the type \eqref{ESQBessel},
where  the Brownian motion $W$ is replaced by a continuous
semimartingale whose martingale component is a Brownian motion.
Then \eqref{lmp2} still holds for every $f \in \shd_{L^{\delta}}$.
\item For $\delta = 0$  and $x_0 =0$, $BESQ^{0}(0)$ is the null process.  By
 Proposition \ref{domain} $L^{0}f(0) = 0$ for all $f \in \shd_{L^0}$,
 obviously $f(0) - f(0) - \int_{0}^{t}L^{\delta}f(0)ds \equiv 0$
and \eqref{lmp2} holds.
\end{enumerate}

\end{remark}

\begin{prooff} (of Proposition \ref{S}).

We consider immediately the case of Remark \ref{RS} (1) and
suppose $W$ to be a semimartingale such that $[W]_t \equiv t$.
 Let $X = \sqrt{S}$, where $S$
is a $BESQ^{\delta}(s_0)$,
let $f \in \shd_{L^{\delta}}$ and
 define $f_n: \mathbb{R}_+ \rightarrow \mathbb{R}$ as $f_n(y) = f\left(\sqrt{y +\frac{1}{n}}\right)$.
Clearly $f_n \in C^2(\mathbb{R}_+)$.
Applying It\^o's formula we have
\begin{align}\label{ES}
  f_n(S_t) = f_n(S_0) &+
                        \int_{0}^{t} \frac{f'\left(\sqrt{S_s +
              \frac{1}{n}}\right)}{\sqrt{S_s + \frac{1}{n}}}\sqrt{S_s}dW_s + \int_{0}^{t}\delta\frac{f'\left(\sqrt{S_s + \frac{1}{n}}\right)}{2 \sqrt{S_s + \frac{1}{n}}} ds \nonumber\\ &+ \int_{0}^{t}\left[\frac{1}{2}f''\left(\sqrt{S_s + \frac{1}{n}}\right) - \frac{1}{2}\frac{f'\left(\sqrt{S_s + \frac{1}{n}}\right)}{\sqrt{S_s + \frac{1}{n}}}\right]\left[\frac{S_s}{S_s + \frac{1}{n}}\right] ds,	
\end{align}
which can be rewritten as
\begin{align}\label{ES2}
f_n(S_t) = f_n(S_0) &+ \int_{0}^{t}\frac{f'\left(\sqrt{S_s + \frac{1}{n}}\right)}{\sqrt{S_s + \frac{1}{n}}}\sqrt{S_s}dW_s + \int_{0}^{t}\frac{1}{2}f''\left(\sqrt{S_s +
\frac{1}{n}}\right)\left[\frac{S_s}{S_s + \frac{1}{n}}\right]ds  \nonumber\\ &+ \frac{1}{2}\int_{0}^{t}\frac{f'\left(\sqrt{S_s + \frac{1}{n}}\right)}{\sqrt{S_s + \frac{1}{n}}}\left[\delta -
\frac{S_s}{S_s + {\dfrac{1}{n}}}\right]ds.	
\end{align}
The first integral converges to
\begin{equation}\label{EFirst}
\int_{0}^{t}f'\left(\sqrt{S_s}\right)dW_s,
\end{equation}
u.c.p. by Remark \ref{r1},
since $f' \in C(\R_+)  $.

As $y \mapsto f''\left(\sqrt{y + \frac{1}{n}}\right)$ is continuous,
 by Remark \ref{r1} and Lebesgue's dominated convergence theorem
 the second integral in
\eqref{ES2}
converges u.c.p. to
\begin{equation} \label{EThird}
 \frac{1}{2}\int_{0}^{\cdot}f''\left(\sqrt{S_s}\right)ds.
\end{equation}
We set $\ell:\R_+ \longrightarrow \R$, the continuous function defined by
$$ \ell(x) = \left\{\begin{array}{ccc}
\dfrac{f'(x)}{x} &:& x \neq 0.\\
f''(0+) &:&  x = 0.
\end{array}\right.$$
The third integral can be rewritten as $$\frac{1}{2}\int_{0}^{t}\ell\left(\sqrt{S_s + \frac{1}{n}}\right)\left[\delta -
\dfrac{S_s}{S_s + {\dfrac{1}{n}}}\right]ds.$$
By Remark \ref{r1} and Lebesgue's dominated convergence the previous expression converges u.c.p.
to \begin{equation}\label{A11}
\int_{0}^{t}\ell\left(\sqrt{S_s}\right)\left(\dfrac{\delta - 1}{2}\right)ds.
\end{equation}
Finally \eqref{EFirst}, \eqref{EThird} and \eqref{A11} allow to conclude
the proof of \eqref{lmp2}.
\end{prooff}	

\begin{corollary}\label{sts1}
Let $x_0 \in \R, 0 \leq \delta < 1$.
The martingale problem with respect to $\shd_{L^\delta}$, with initial
condition $X_0 = x_0$ admits strong existence.
More precisely we have the following.
If $x \ge 0$, we denote by $X^{x}$ the $BES^{\delta}(x)$ process, being the square root of a
solution of \eqref{ESQBessel} with $s_0 = x^2$.
\begin{enumerate}
\item
If $x_0 \ge 0$, $X^{x_0}$
solves the strong martingale problem with respect to
 $\shd_{L^{\delta}}$ and $W$.
\item If  $x_0 \le 0$, $- X^{-x_0}$
  solves the same
strong martingale problem with respect to
 $\shd_{L^{\delta}}$ and $- W$.
\end{enumerate}
\end{corollary}
\begin{preuve}

Let $(\Omega, \mathcal{F}, \P)$ be  a  probability space and a Brownian motion $W$.
We set $s_0 = x_0^2$. We know that
\eqref{ESQBessel} admits a strong solution $S$.
  Then, by
 Proposition \ref{S} $X = \sqrt{S}$ is a solution for
the strong martingale problem with respect to $\shd_{L^\delta}$ and $W$
with initial condition $\vert x_0 \vert$.

 So, if $x_0 \ge 0$ then
 strong existence is established. If $x_0 < 0$ then
we show below that $-X$ also solves the strong martingale problem
with respect to $\shd_{L^\delta}$ and $-W$.

Let $f \in \shd_{L^{\delta}}$. Then
obviously $f_{-}(x) := f(-x) \in \shd_{L^\delta}$
and $$L^{\delta} f_{-}(x) = L^{\delta} f(-x).$$
Therefore, since $X$ solves the strong martingale problem
with respect to $\shd_{L^\delta}$ and $W$,
 for all $f \in \shd_{L^{\delta}}$ we have
$$ {f}_{-}(X_t) - f_{-}(x_0) - \int_{0}^{t}L f_{-}(X_s)ds =  \int_0^t f_{-}'(X_s)
 dW_s, $$
 which implies
$$
f(-X_t) - f(-x_0) - \int_{0}^{t}L^{\delta}f(-X_s)ds = \int_0^t f'(-X_s) d(-W)_s.$$
  Thus $-X$ also solves the strong martingale problem
with respect to $\shd_{L^{\delta}}$ and $- W$.

\end{preuve}	

\begin{proposition}\label{nu}
Let us suppose 	$0 < \delta < 1$.
The martingale problem with respect to $\shd_{L^{\delta}}$
 does not admit (in general) uniqueness in law.
\end{proposition}

\begin{preuve}
	
Let $S$ be the $BESQ^{\delta}(0)$. By Corollary
\ref{sts1}, we know that $X^+ = \sqrt{S}$ and $X^- = - \sqrt{S}$
solve the martingale problem with respect to an underlying
probability $\P$.

Obviously $X$ does not have the same law as $-X$ since $X$ is positive
and $-X$ is negative.
\end{preuve}

\begin{remark} \label{ROtherConditions}
If the initial condition $x_0$ is different from zero, for instance positive,
then uniqueness also fails since we can exhibit two solutions.
The first one is still the classical Bessel process, the second one
behaving as the first one until it reaches zero and then it behaves like
minus a Bessel. Such a stopping time always exists since
the Bessel process hits zero, see the considerations after Corollary (1.4)
in Chapter XI in \cite{Yor}.
\end{remark}
For proving indeed results for uniqueness, we will need the following.
\begin{proposition}
\label{PAllBessel}
 Let $0 \le \delta < 1$.
Let $(X, \P)$ be a solution (not necessarily positive) of
the
 martingale problem with respect to $\shd_{L^\delta}$.
Then $S = X^2$ is a squared Bessel process.

\end{proposition}
\begin{preuve}
	
We first show that
 \begin{equation}\label{Mt}
M^1_t := X_t^{2} - x_0^2 - \delta t
\end{equation}
 is a local martingale and
  \begin{equation}\label{e3}
 X^4_t = x^4_0 + 2(2 + \delta)\int_{0}^{t}X^2_s ds + M^2_t,
 \end{equation}
where $M^2$ is a local martingale.
Clearly, $f_1(x) :=  x^2 \in \shd_{L^{\delta}}$
because $f \in C^2(\R)$ and
$f_1'(0) = 0$.
By Proposition \ref{domain}
 $L^{\delta}f_1(x) \equiv \delta,$ which shows \eqref{Mt}.
On the other hand, obviously
$f_2(x) := x^4 \in \shd_{L^{\delta}}$
and then, by Proposition \ref{domain},
 $L^{\delta}f_2(x) = 2(2 + \delta)x^2$, so  \eqref{e3} follows.
 Now, setting $S := X^{2}$, by integration by parts
and using \eqref{Mt} we have
\begin{equation}\label{e2ter}[M^1]_t = [S]_t = S^{2}_t - s^2_0 - 2\int_{0}^{t} S_s dS_s = X^4_t - x^4_0 - 2\delta\int_{0}^{t}X_{s}^2ds + M_t,\end{equation}
where $M$ is a local martingale.
This implies
\begin{equation}\label{e2bis}
X^4_t = x^4_0 + 2\delta\int_{0}^{t}X_{s}^2ds + [M^1]_t -  M_t.
\end{equation}
We remark that \eqref{e2bis} and \eqref{e3} provide two decompositions
of the semimartingale $X^4$. By uniqueness of the semimartingale decomposition
we can identify the bounded variation component, which implies
 \begin{equation} \label{E[M]}
[M^1]_t = 4\int_{0}^{t}X^2_s ds, t \in  [0,T].\end{equation}
Consequently the process
 $$W_t:= \int_{0}^{t}\dfrac{dM^1_s}{2|X_s|},$$ is a Brownian motion
taking into account the fact that $[W]_t \equiv t$
together with L\'evy's characterization of Brownian motion.
 Hence, the process $S$ is a (weak) solution of the
 SDE \begin{equation}\label{dS}dS_s = \delta ds + 2\sqrt{|S_s|}dW_s,
\end{equation}
 which shows that $S$ is a $BESQ^{\delta}(s_0)$, $s_0 = x^2_0$.

\end{preuve}


Proposition \ref{nu} shows that no uniqueness on the real line holds when $\delta > 0$. Surprisingly, if $\delta = 0$ then uniqueness holds.
\begin{remark} \label{delta=0}
Suppose $\delta = 0$.
\begin{enumerate}
\item Assume $x_0 = 0$. By Proposition \ref{PAllBessel}
if $(X, \P)$ is a solution of the martingale problem, then $X^2$ is (under $\P$) a
$BESQ^0(0)$ which is the null process; this fact
 shows uniqueness.
\item Suppose $x_0$ different from zero
(for instance strictly positive). If $(X,\P)$
is a solution to the strong martingale problem,
then, by Proposition \ref{PAllBessel}, under $\P$,
 $S:=X^2$ is a $BESQ^{0}(x^2_0)$.

In particular $S$ is a solution
of \eqref{ESQBessel} with respect to some suitable Brownian motion $W$.
Then, the strong Markov property shows that, whenever $S$
reaches zero it is forced to remain there.
\end{enumerate}
\end{remark}

At the level of strong martingale problem we have the following.

\begin{proposition}
\label{PAllBesselStrong}
 Let $0 \le \delta < 1$.
Let $X$  be a non-negative solution  to the strong martingale problem
with respect to $\shd_{L^\delta}, \sigma$ and a Brownian motion $W$.
Then $S = X^2$ is a solution to \eqref{ESQBessel}.
\end{proposition}

 \begin{preuve}
	
   Let us suppose that $X$ is a solution of the strong martingale problem with respect to $ \shd_{L^{\delta}}$
   and a Brownian motion $W.$
   Setting $S := X^2$ and applying \eqref{lmpBis} with $f_1(x) = x^2$ we get
   $$ S_t = s_0 + 2 \int_0^t \sqrt{|S_s|} dW_s + \delta t, t \in [0,T], $$
   with $s_0 = x_0^2. $
 \end{preuve}	

\subsection{The martingale problem in the
$\mathbb{R}_+$-case}

\label{S54}

$  $

We remain still with the case $0 \le \delta < 1$.
Let $(\Omega, \mathcal{F}, \P)$ be a probability space and a Brownian motion $W$.
We will  be interested in non-negative solutions $X$ for the strong martingale problem
with respect to
$\shd_{L^{\delta}}(\mathbb{R_+})$ and $W$, which means that
	\begin{equation}\label{lmp3}
	f(X_t) - f(X_0) - \int_{0}^{t}L^{\delta}f(X_s)ds  = \int_{0}^{t}f'(X_s)dW_s,
	\end{equation}
for all $f \in \shd_{L^{\delta}}(\mathbb{R}_+)$.
Proposition \ref{P54} below states the existence result.
It follows directly from the $\mathbb{R}$-case, see Proposition \ref{S}.
\begin{proposition} \label{P54} Let $0 \leq \delta < 1$.
The process  $BES^{\delta}(x_0)$ as stated in Proposition
\ref{S} solves the strong martingale
 problem with respect to
$\shd_{L^{\delta}}(\mathbb{R_+})$ and $W$.
In particular, the martingale problem related to $\shd_{L^{\delta}}(\R_+)$ admits strong existence.
\end{proposition}

\begin{proposition}\label{su}
	The martingale problem with respect to
$\shd_{L^{\delta}}(\mathbb{R_+})$ and $W$
admits pathwise uniqueness.
\end{proposition}

\begin{preuve}

Let us suppose that $(X, \P)$ is a  solution of the  martingale problem
with respect to $\shd_{L^{\delta}}(\mathbb{R_+})$
and $W$.
This implies the same with respect to $\shd_{L^{\delta}}$.
By Proposition \ref{PAllBesselStrong} $S = X^2$
is a solution of \eqref{ESQBessel} for some Brownian motion $W$.
The result follows by the pathwise uniqueness of the SDE
\eqref{ESQBessel} and the positivity of $X$.

\end{preuve}
\subsection{The martingale problem related to an extended domain}
\label{ExtDomain}

In this section we answer to the question raised in Remark
\ref{Rharmonic}. Indeed, for some aspects, one could be interested
in a formulation of the martingale problem with respect to the extended domain $\hat{D}_{L^{\delta}}$
defined in Remark \ref{Rharmonic} in order to include the
harmonic function \eqref{Eharmonic}.

\begin{proposition}\label{extdom}
	Let $(X_t)_{t\geq 0}$ be a solution to the martingale problem with respect to $\shd_{L^\delta}$. Suppose the following.
	\begin{enumerate}[label=\roman*)]
		\item For almost all $t\in \ ]0,T]$ the law of $X_t$ admits density $p_t.$
		\item $\displaystyle \lim_{|x|\to 0}\int_{0}^{T}|x|^{1-\delta}p_t(x)dt = 0.$
	\end{enumerate}
Then $(X_t)$ is also a solution to the martingale problem with respect to $\hat{\shd}_{L^{\delta}}$.
\end{proposition}
\begin{remark}\label{rextdom}
  An analogous statement is valid for the strong martingale problem.
\end{remark} 
\begin{prooff} (of Proposition \ref{extdom}).
\
Let $f\in \hat{D}_{L^{\delta}}$ and  consider a smooth bounded function $\chi:\R\longrightarrow\R_+$ such that
\begin{equation}\label{Echi1}
\chi(x)= \left \{
\begin{array}{cc}
	1,  & x\leq -1 \\
	0, & x \geq 0 \\
	S(x), & x \in\ [0,1],
\end{array}
\right .
\end{equation}
for some bounded function $S:\ [-1,0]  \longrightarrow \ [0,1]$
with $S(0) = 0, S(-1) = 1$.
For every $n \geq 1$ we define $\chi_n:\R\longrightarrow\R_+$ as
\begin{equation*}
	\chi_n(x):=\chi\left(\dfrac{1}{2}-n|x|\right).
\end{equation*}
Notice that
\begin{equation*}
	\chi_n(x) = \left \{
	\begin{array}{cc}
		0,  & |x|\leq \frac{1}{2n} \\
		1, & |x| \geq \frac{3}{2n} \\
		\in [0,1], & \text{otherwise.} 
	\end{array}
	\right .
\end{equation*}

We have $\chi_n'(x) = \chi'\left(\dfrac{1}{2}-n|x|\right)(-n\text{sign}(x))$, so that
\begin{equation}\label{EA1}
  |\chi'_n(x)|\leq n||\chi'||_{\infty}\mathit{I}_{\{\frac{1}{2n}
    \le |x| \le \frac{3}{2n}\}}(x), \ x \in \R.
\end{equation}
For every $n\geq 1$ we define $f_n:\R\longrightarrow\R_+$ such that 
\begin{equation} \label{EChi-n}
	\left \{
	\begin{array}{cc}
		f_n(0) = f(0)\\
		f'_n=f'\chi_n. \\
	\end{array}
	\right .
\end{equation}
Clearly $f_n \in \shd_{L^{\delta}}$, so
\begin{equation}\label{EA2}
	f_n(X_t) - f_n(X_0) - \int_{0}^{t}L^{\delta}f_n(X_s)ds
\end{equation}
is a local martingale. Obviously $f_n \rightarrow f$ and $f_n' \rightarrow f'$ uniformly on each compact. We show below that
\begin{equation}\label{EA3}
	\int_{0}^{\cdot}L^{\delta}f_n(X_s)ds\xrightarrow{u.c.p.} \int_{0}^{\cdot}L^{\delta}f(X_s)ds.
\end{equation}

By \eqref{DefLForm} and \eqref{EChi-n} we get
\begin{equation*}
	L^{\delta}f_n(x) = \dfrac{|x|^{1-\delta}}{2}(|x|^{\delta - 1}f_n')'(x) = \chi_n(x)L^{\delta}f(x) + \frac{1}{2} \chi_n'(x)f'(x).
\end{equation*}
Since $\chi_n$ converges to $1$ uniformly on each compact,
then $ \int_0^\cdot \chi_n(X_s) L^{\delta}f(X_s) ds$ converges
u.c.p. to $ \int_0^\cdot  L^{\delta}f(X_s) ds$.
To prove \eqref{EA3} it remains 
to prove that
\begin{equation}\label{chi-ucp}
  \int_{0}^{t}\chi_n'(X_s)f'(X_s)ds \xrightarrow{u.c.p.} 0.
\end{equation}
  For this, by \eqref{EA1} we have
\begin{align}\label{EA3Bis}
	&\E\left(\sup_{t\leq T}\left|\int_{0}^{t}\chi'_n(X_s)f'(X_s)ds\right|\right) \leq \E\left(\int_{0}^{T}\left|\chi'_n(X_s)f'(X_s)\right|ds\right) \leq\\\nonumber & \leq n||\chi'||_{\infty}\E\left(\int_{0}^{T}|f'(X_t)|\mathit{I}_{\{\frac{1}{2n}|X_t|<\frac{3}{2n}\}}(X_t)dt\right) = n||\chi'||_{\infty}\int_{0}^{T}\int_{\frac{1}{2n}}^{\frac{3}{2n}}|f'(x)|p_t(x)dxdt.
\end{align}
Let $g$ be the continuous functions such that for $x\neq 0$ we have $g(x) =
f'(x)|x|^{\delta-1}$.
\eqref{EA3Bis} gives
$$
 I(n):= n||\chi'||_{\infty}\int_{0}^{T}\int_{\frac{1}{2n}}^{\frac{3}{2n}}|f'(x)|p_t(x)dxdt = n||\chi'||_{\infty}\int_{0}^{T}\int_{\frac{1}{2n}}^{\frac{3}{2n}}
  |g(x)||x|^{1-\delta}p_t(x)dxdt.
  $$
  Let $\varepsilon >0$.  Taking into account
  hypothesis ii) in the statement, there exists $A>0$ such that for $|x| \le A$,
such that
$\int_{0}^{T}|x|^{1-\delta}p_t(x)dt <  \varepsilon$.
Consequently, for $ \vert x \vert \le A$ 
$$ I(n) \le   \dfrac{3}{2} ||\chi'||_{\infty} \sup_{\vert x\vert \le A}  \vert g(x) \vert \varepsilon.$$
Taking the $\limsup$ when $n$ goes to infinity and since $\varepsilon$ is arbitrary
we get $\limsup_{n \rightarrow +\infty} I(n) = 0$ and consequently \eqref{chi-ucp}.



Since the space of local martingales is closed under the u.c.p. convergence then,
taking the limit on \eqref{EA2} when $n\to \infty$, we conclude that $f(X_s) - f(X_0) - \int_{0}^{t}L^{\delta}f(X_s)ds$ is a local martingale.
\end{prooff}
\begin{proposition}\label{Besatisfy}
Let $(X_t)$ be the Bessel process of dimension $\delta \in\ [0,1]$ starting from $x_0>0$. Then the following holds.
\begin{enumerate}[label=\roman*)]
	\item For every $t>0$ the law of $X_t$ admits a density $p_t$.
	\item $\displaystyle \lim_{\vert x \vert \to 0^+}\int_{0}^{T}\vert x\vert ^{1-\delta}p_t(x)  = 0$. 
\end{enumerate}
\end{proposition}
Before doing the proof we recall that $I_{\nu}$ the modified Bessel function of first kind (see \cite{abra}, section 10) with $\nu = \dfrac{\delta}{2} - 1$. To prove Proposition
\ref{Besatisfy} we will make use of the estimate stated in the following lemma.
\begin{lemma}\label{Befuncbound}
  $I_{\nu}(z) \leq C\exp(z)$, for some constant $C$ and $z\in \R$
  large enough.
\end{lemma}
\begin{proof}
  In \cite{abra} equation 9.6.20 (p.376) we have
  $$I_{\nu}(z) = \dfrac{1}{2\pi}\int_{0}^{\pi}\exp(z\cos(\theta))
  \cos(\nu \theta z) d\theta - \dfrac{\sin(\nu\pi)}{\pi}
  \int_{0}^{\infty}\exp(-z\cosh(t) - \nu t)dt=: I_1(z) - I_2(z). $$
  For $z > 0$ we get
  $$ \vert I_1(z) \vert \le \frac{1}{2} \exp(z). $$
  Concerning $I_2(z)$  we first observe that
$ - z \cosh(t) - \nu t \le (- z - \nu) t$ for $t \ge 0$.
Let $R > -\nu  $. For $z > R$ we get
$$\vert I_2(z) \vert \le \dfrac{1}{\pi} \int_0^\infty \exp(-t(R + \nu)) dt
=    \dfrac{1}{\pi (R+\nu)}.$$
Consequently the result follows.
\end{proof}
\begin{prooff} (of Proposition \ref{Besatisfy}).
\begin{enumerate}[label=\roman*)]
\item
  We recall (see \cite{Jean}, chapter 6 equation 6.2.2 and Appendix A)
  that for $X_0 = x_0$
	\begin{equation}\label{EB1}
		p_s(y) = \dfrac{y}{s}\left(\dfrac{y}{x_0}\right)^{\nu}\exp\left(-\frac{x^2_0 + y^2}{2s}\right)I_{\nu}\left(\dfrac{x_0y}{s}\right),
	\end{equation}
      \item Since $X$ is non-negative, we can remove the absolute value
        from $\vert x \vert$.
        By \eqref{EB1} and Lemma \ref{Befuncbound} we have
\begin{align*}
  &x^{1-\delta}\int_{0}^{T}p_{t}(x) dt  \le C
    \dfrac{x^{1-\delta + 1+\frac{\delta}{2}-1}}{x_0^{\frac{\delta}{2} -1}}\int_{0}^{T}\exp\left(-\dfrac{x^2_0 + x^2}{2t}+\dfrac{x_0x}{t}\right)\dfrac{1}{t}dt \\
  &\le C x^{1-\frac{\delta}{2}} x_0^{1-\frac{\delta}{2}}\int_{0}^{T}\exp\left(-\dfrac{(x_0 - x)^{2}}{2t}\right)\dfrac{1}{t}dt.
\end{align*}
For $t> 0$ and $x < x_0$ we set $\tilde{t}:=\dfrac{(x_0 - x)^2}{t}$, so $dt = - \dfrac{(x_0-x)^{2}}{\tilde{t}^2}$. That gives us
\begin{equation}\label{EIntegral}
C(x x_0)^{1- \frac{\delta}{2}}
  \int_{\frac{(x_0 -x)^2}{T}}^\infty \exp\left(\dfrac{-\tilde{t}}{2}\right)
  \dfrac{1}{\tilde{t}}d\tilde{t}.
  \end{equation}
  Since previous integral converges to
  $$ \int_{\frac{x_0^2}{T}}^\infty \exp\left(\dfrac{-\tilde{t}}{2}\right),$$
when $x \rightarrow 0$, then \eqref{EIntegral} converges to zero.
So the proof is concluded.
\end{enumerate}
\end{prooff}
\begin{remark} \label{R317} 
We remark that item (ii) of Proposition \ref{Besatisfy} is not
fulfilled for a Bessel process starting from $x_0=0$,
see Proposition \ref{P319}. In this case, if one replaces the initial domain
$\shd_{L^\delta}$ with its extended domain
the Bessel process fulfills a martingale problem where
one has to add a supplementary term in the operator $L^\delta$.
This research is developed in an ongoing draft, which
goes beyond the scope of the present paper.

\end{remark}
\begin{proposition} \label{P319}
  Let $(X_t)$ be the Bessel process with dimension $\delta \in  [0,1]$
  starting at $x_0=0$. Then, the following holds.
\begin{enumerate}[label=\roman*)]
	\item For every $t>0$ the law of $X_t$ admits a density $p_t$
	\item For every $t>0$
          $\displaystyle \lim_{x\to 0^+}\int_{0}^{t}x^{1-\delta}p_s(x)ds =
      \dfrac{2^{2-\frac{\delta}{2}}}{\Gamma(\frac{\delta}{2})}t^{1-\frac{\delta}{2}}\dfrac{1}{2-\delta}.$
\end{enumerate}
where $\Gamma$ is the Gamma function given by $\Gamma(a) = \int_{0}^{\infty}x^{a-1}\exp(-x)dx, a> 0$.
\end{proposition}
\begin{proof}
	According to equation 6.2.2 in \cite{Jean} we have 
\begin{equation*}
	p_t(x)= \dfrac{2^{\nu}t^{-(\nu+1)}}{\Gamma(\nu+1)}x^{2\nu +1}\exp\left(-\dfrac{x^2}{2t}\right).
\end{equation*}
Consequently, since $\nu = \frac{\delta}{2} - 1$, we get
\begin{equation*}
	x^{\delta - 1}\int_{0}^{t}p_{s}(x)ds = \dfrac{2^{1 - \frac{\delta}{2}}}{\Gamma(\frac{\delta}{2})}\int_{0}^{t}\exp\left(-\frac{x^2}{2s}\right)s^{-\frac{\delta}{2}}ds
\end{equation*}
For $s>0$ and $x>0$ we set $\tilde{s} = \dfrac{x^2}{s}, ds =-\dfrac{x^2}{\tilde{s}^2}d\tilde{s}$. We obtain
\begin{align*}
	&\dfrac{2^{1 - \frac{\delta}{2}}}{\Gamma(\frac{\delta}{2})}\int_{\frac{x^2}{t}}^{\infty}\dfrac{x^2}{\tilde{s}^2}\left(\dfrac{\tilde{s}}{x^2}\right)^{\frac{\delta}{2}}\exp\left(-\dfrac{\tilde{s}}{2}\right)d\tilde{s} = \dfrac{x^{2-\delta}2^{1-\frac{\delta}{2}}}{\Gamma(\frac{\delta}{2})}\int_{\frac{x^2}{t}}^{\infty}\tilde{s}^{-(\frac{\delta}{2})}\exp\left(\dfrac{\delta}{2}\right)d\tilde{s} =\\ &=\dfrac{2^{1-\frac{\delta}{2}}}{\Gamma(\frac{- \tilde s}{2})}\dfrac{1}{x^{\delta-2}}\int_{\frac{x^2}{t}}^{\infty}\tilde{s}^{\frac{\delta}{2}-2}\exp\left(-\dfrac{\tilde{s}}{2}\right)d\tilde{s}
\end{align*}
Since the integral and $\dfrac{1}{x^{\delta-2}}$ go to $\infty$ when $x\to 0^+$ then, by L'Hospital rule,
\begin{align*}
	&\displaystyle\lim_{x\to 0^+}\int_{0}^{t}x^{1-\delta}p_s(x)ds = -\dfrac{2^{1-\frac{\delta}{2}}}{\Gamma(\frac{\delta}{2})}\lim_{x\to 0^+}\dfrac{\frac{2x}{t}(\frac{x^2}{t})^{\frac{\delta}{2}-2}\exp(-\frac{x^2}{t})}{(\delta-2)x^{\delta-3}} =\\ &= -\dfrac{2^{2-\frac{\delta}{2}}}{\Gamma(\frac{\delta}{2})}t^{1-\frac{\delta}{2}}\dfrac{1}{\delta -2}\lim_{x\to 0^+}x^{0} = \dfrac{2^{2-\frac{\delta}{2}}}{\Gamma(\frac{\delta}{2})}t^{1-\frac{\delta}{2}}\dfrac{1}{2-\delta}
\end{align*}
\end{proof}

\subsection{On an alternative approach to treat
the martingale problem on the full line.}

\label{R55}

$ $

A priori we could have approached the martingale problem
related to Bessel processes by the technique of \cite{frw1}.

\begin{enumerate}
  \item
  Thereby, the authors handled martingale problems related
to operators $L: \shd_L \subset C^1(\R) \rightarrow \R$
of the form $L f = \frac{\sigma^2}{2} f'' + b'f',$ where
$b$ is the derivative of a continuous function, $\sigma$
  is strictly positive continuous and
  $\Sigma$ is defined as \eqref{Sigmaintr}.
  The idea was to consider an
  $L$-harmonic function $h:\R \rightarrow \R$
  defined by
  $h(0) = 0$
and $h' = e^{-\Sigma}$.
In \cite{frw1}, $L$ was also expressed
 in the form \eqref{DefLbis}. The proof of well-posedness
of the martingale problem thereby
was based on a non-explosion condition
(3.16)   in Proposition
3.13 in \cite{frw1}
and
the fact that
$\sigma_0:= (\sigma e^{-\Sigma}) \circ h^{-1}$
is strictly positive and so the SDE (for every fixed initial condition)
\begin{equation}\label{EY}
  Y_t = y_0 + \int_0^t \sigma_0(Y_s) dW_s,
  \end{equation}
is well-posed.
\item 	Consider $\delta \in [0, 1[$.
  As far as the martingale problem (for the Bessel process)
  on the full line  is concerned, we could have tried to adapt
  similar methods.
  We observe that $L := L^\delta$ is also expressed in the form
  \eqref{DefLbis}, which in our case gives \eqref{DefLForm}.
  Taking into account \eqref{hBessel}, we have
  \begin{equation} \label{Eh}
    h(x) = \sgn(x) \frac{{\vert x \vert}^{2 - \delta}}{2-\delta},  x \in \R.
    \end{equation}
Since $h$ is bijective, one can show that
(3.16)   in Proposition
3.13
in \cite{frw1}
is automatically satisfied. Moreover
\begin{equation} \label{Esigma0}
\sigma_0(y) = \sgn(y)(2 - \delta)^{\frac{1 - \delta}{2 -
   \delta}}|y|^{\frac{1 - \delta}{2 - \delta}}.
\end{equation}
Following the same idea as in in Proposition 3.2 of \cite{frw1},
one can show that the well-posedness of the Bessel martingale problem
(with respect to $\hat \shd_{L^\delta}$)
is equivalent to the well-posedness (in law) of
\eqref{EY}.
Here $\sigma_0(0) = 0$, but
 \eqref{EY} is still well-posed even if
\begin{equation}\label{ESchmidt}
\int_0^{\varepsilon} \frac{1}{\sigma_0^2}(y) dy = +\infty, \
 \forall \varepsilon > 0.
\end{equation}
In fact in that case \eqref{ESchmidt} corresponds to the
Engelbert-Schmidt criterion (see Theorem 5.7 in
\cite[Chapter 5]{ks}.
\item The criterion \eqref{ESchmidt} can be reformulated here
  saying
that the quantity
\begin{equation} \label{EDiver}
\frac{1}{(2 - \delta)^{\frac{2 - 2\delta}{2 - \delta}}}
 \int_{0}^{\epsilon}y^{\frac{2\delta - 2}{2-\delta}}dy,  \
 \forall \varepsilon > 0,
\end{equation}
is infinite.
Now, \eqref{EDiver} is always finite for any $\delta > 0$.	
	This confirms that \eqref{EY}
        has no uniqueness in law on $\R$, with $\sigma_0$
defined in \eqref{Esigma0},
        when $\delta \in ]0,1[$.
 So, the non-uniqueness observed
in Proposition \ref{nu} is not astonishing.
\item
On the other hand, when $\delta = 0$, then \eqref{EDiver} is infinite,
which implies uniqueness in law.
\item We drive the attention on the fact
  that the considerations of this section
  concern the martingale problem with respect
  to the extended domain $\hat \shd_{L^\delta}$ and
for the case $x_0 \neq 0$. 
\end{enumerate}

\subsection{The framework for $\delta = 1$}
 $ $

\label{S55}

Let $W$ be a standard Brownian motion on some underlying
probability space.
By definition, a Bessel process of dimension $\delta = 1$ starting at
$x_0 \ge 0$ is
a non-negative process $X$ such that $S:=X^2$ is a $BESQ^1(x_0^2)$.
On the other hand, in the literature such a Bessel process
$X$ is also characterized as a non-negative strong solution of
\begin{equation} \label{LocTime}
 X_t = x_0 + W_t + L_t, t \in [0,T],
\end{equation}
where $L$ is
 a non-decreasing process only increasing when $X = 0$, i.e.
$$ \int_{[0,T]} X_s dL_s = \int_{[0,T]} X_s 1_{\{X_s = 0\}} dL_s.$$
In particular, $X$ is a semimartingale.
Indeed, let $X$ be a non-negative solution of \eqref{LocTime}, then
by an easy application of It\^o's formula for semimartingales, setting $S:= X^2$, we have
\begin{eqnarray*}
 S_t &=& x_0^2 + 2 \int_0^t X_s dW_s +  \int_0^t X_s  dL_s +
\frac{1}{2} 2t \\
&=& x_0^2 + 2 \int_0^t \sqrt{S_s} dW_s +  \int_0^t X_s1_{\{X_s = 0\}} dL_s + t\\
&=& x_0^2 + 2 \int_0^t \sqrt{S_s} dW_s  + t,
\end{eqnarray*}
which implies that $S$ is a $BESQ^1(x_0^2)$ and so $X$
is a $BES^1(x_0)$.
This shows in particular that \eqref{LocTime} admits pathwise
uniqueness.
Existence and uniqueness of \eqref{LocTime} can be seen
via the Skorohod problem, see \cite{harri}.

In this section, we represent alternatively $X$ as a non-negative solution
 of a (strong) martingale problem.
As we mentioned at the beginning of Section \ref{SBessel}, we have fixed
$$ b(x) = H(x) = \left \{
\begin{array}{ccc}
1 & : & x \ge 0\\
0  & : & x < 0.
\end{array}
\right.
$$
Formally speaking we get $$\Sigma(x) = 2\int_{0}^{x} \delta_{0}(y) dy = 2H(x),$$
where $H$ is the Heaviside function. Coming back to the expression
\eqref{DefLbis},
it is natural to set
\begin{equation} \label{EL1}
L^1f = (\exp(2H)f')'\frac{\exp(-2H)}{2}, \ f \in C^2(\R^*).
\end{equation}
This gives of course
\begin{equation} \label{EL1bis}
L^1f = \frac{f''}{2}, \ f \in C^2(\R^*).
\end{equation}
Analogously to the case $\delta \in ]0, 1[$
and applying the same principle as for the domain characterization in the
case $\delta \in [0,1[$, we naturally arrive to
$$\shd_{L^1}= \{ f \in C^2 \vert f'(0)= 0\}.$$
Since $L^1f$ has to be continuous, \eqref{EL1bis} gives
      \begin{equation}\label{EL2}
             L^{1} f =
             \frac{f''}{2}.
             \end{equation}
The PDE operator $L^1$ appearing at \eqref{EL2} coincides with
 the generator of Brownian motion. However,
the domain of that generator is larger since it is $C^2(\R)$.

\begin{remark} \label{delta= 1}
The same preliminary analysis of Section \ref{S53} about the
martingale problem related to $ 0 \le \delta < 1$
in the $\R$-case extends to the case  $\delta = 1$.
More precisely, Proposition \ref{S}, Corollary \ref{sts1}, Proposition \ref{nu}
and Remark \ref{ROtherConditions} hold. This is stated below.
\end{remark}
\begin{proposition} \label{Pdelta=1}
	$ $
	
\begin{enumerate}
\item
There is a  process $BES^{1}(x_0)$
  solving the strong martingale
 problem with respect to
 $\shd_{L^{1}}$ and $W$.
\item The martingale problem related to $L^1$ with respect
to $\shd_{L^1}$ admits (in general) no uniqueness.
\end{enumerate}
\end{proposition}
Similarly to Corollary \ref{sts1}, the processes $BES^1(x_0)$ and
 $ - BES^1(-x_0)$ are solutions to the strong martingale problem
with respect to  $\shd_{L^1}$ and an underlying Brownian motion $W$.
Other solutions on the real line
are the so-called
 {\it skew Brownian motions}
 which will be investigated more in detail
in a  future work.
 For this last one, we can mention the works of Harrison and Shepp (\cite{harri}) and Le Gall (\cite{LeGall}).

Concerning the $\R_+$-case, let again $(\Omega, \shf, \P)$ be a probability space equipped with the canonical filtration
 ${\mathfrak F}^W$ of a Brownian motion $W$.

By using the same arguments as for Propositions \ref{P54} and \ref{su},
we get the following result.
\begin{proposition}\label{S1}
  There is a  process $BES^{1}(x_0)$
  solving the strong martingale
 problem with respect to
 $\shd_{L^{1}}(\mathbb{R_+})$ and $W$.
 Moreover, the martingale problem admits
 pathwise uniqueness with respect to $\shd_{L^{1}}(\mathbb{R_+})$.
\end{proposition}

\section{Martingale problem related to
  the path-dependent Bessel process}

\label{SBesselPath}

\subsection{Generalities}
$ $

Now we are going to treat a non-Markovian martingale problem
which is a perturbation of the Bessel process
 $BES^{\delta}(x_0)$, $0\leq \delta \leq 1, x_0 \ge 0$. More precisely, we want to analyze existence and uniqueness of
solutions to the martingale problem related to the SDE
\begin{equation}\label{YB}
X_t = x_0 + W_t + \int_{0}^{t}b'(X_s)ds + \int_{0}^{t}\Gamma(s, X^{s})ds,
\end{equation}
where $\Gamma$ is the same path-dependent functional as in \eqref{EGamma}, and
$b$ is as in \eqref{EbBessel}.
\begin{proposition} \label{PNonMark_0}
Suppose $\delta = 0, x_0 = 0$. Let $W$ be a standard Brownian motion.
 The null process is a solution to the strong martingale problem
(in the sense of Definition \ref{DSolution})
with respect to $\shd_{L^\delta}$ and $W$.
\end{proposition}
 In presence of a path-dependent drift $\Gamma$, under suitable conditions,  Corollary \ref{Solution}
allows to show that the null process is still the unique solution of the
corresponding strong martingale problem.

\subsection{The martingale problem in the path-dependent case:
 existence in law.}

$ $

We recall  that a pair $(X, \P)$ is a solution for the martingale problem
related to $\shl$ in the sense of Definition \ref{D31}
with $L = L^\delta$ with respect to
 $\shd_{L^{\delta}}$
(resp. $\shd_{L^{\delta}}(\R_+)$), $0\leq \delta \leq 1$, if
for all $f \in \shd_{L^{\delta}}$
(resp. $f \in \shd_{L^{\delta}}(\R_+)$),
\begin{equation}\label{lmp4}
f(X_t) - f(X_0) - \int_{0}^{t}L^{\delta}f(X_s) ds - \int_{0}^{t}f'(X_s)
\Gamma(s, X^{s})ds,
\end{equation}
is a $\P$-local martingale.

A first criterion of  existence can be stated
if $\Gamma$ is measurable and bounded.


\begin{proposition} \label{P61}
Suppose that $\Gamma$ is bounded.
 Then the martingale problem related to $\shl$ (defined in  \eqref{DOperatorL})
admits existence
 with respect to $\shd_{L^{\delta}}$.
Moreover we have the following.
\begin{enumerate}
\item
If the initial condition is $x_0 \ge 0$, then
the solution can be constructed to be non-negative.
\item If the initial condition is $x_0 \le 0$, then
the solution can be constructed to be non-positive.
\end{enumerate}

\end{proposition}
\begin{preuve} \
Let
  $x_0 \geq 0$.
  Given a Brownian motion $W$,
  by Propositions \ref{P54}
and \ref{Pdelta=1},
    there exists a solution $X$ to the (even strong) martingale problem
related to \eqref{DOperatorL} (with $\Gamma = 0$)
 with respect to $\shd_{L^{\delta}}(\R_+)$ and $W$. That solution
 is in fact a $ BES^{\delta}(x_0)$.
  In particular, for all $f \in \shd_{L^{\delta}}(\R_+)$,
\begin{equation}\label{e3bis}
f(X_t) - f(X_0) - \int_{0}^{t}L^{\delta}f(X_s)ds = \int_{0}^{t}f'(X_s)dW_s.
\end{equation}
Since the Bessel process is non-negative, \eqref{e3bis}
also holds for  $f \in \shd_{L^{\delta}}$.
As $\Gamma$ is bounded then,
 by Novikov's condition $$N_t = \exp\left(\int_{0}^{t}\Gamma(s, X^s)dW_s - \frac{1}{2}\int_{0}^{t}\Gamma^2(s, X^s)ds\right),$$
is a martingale. By Girsanov's Theorem
$$B_t := W_t - \int_{0}^{t}\Gamma(s, X^s)ds,$$
is a Brownian motion under the probability measure $\Q$ such that
$d\Q = N_T d\P$.
Then, we can rewrite \eqref{e3bis} as
$$f(X_t) - f(X_0) - \int_{0}^{t}L^{\delta}f(X_s)ds - \int_{0}^{t}f'\left(X_s\right)dB_s - \int_{0}^{t}f'\left(X_s\right)\Gamma(s, X^s)ds = 0.$$
Since
$\displaystyle\int_{0}^{t}f'\left(X_s\right)dB_s$ is a $\Q-$local martingale,
$(X, \Q)$ happens to be a solution to the
martingale problem in the sense of Definition \ref{D31} with respect
to $\shd_{L^\delta}$.

Suppose now that $x_0 \le 0$. The process   $X$
defined as  $ - BES^{\delta}(-x_0)$
 is a solution
of \eqref{e3bis}. Then the same procedure as for the case $x_0 \ge 0$
works. This shows existence for the martingale problem
on $\shd_{L^\delta}$.

Let us discuss the sign of the solution.
Suppose that $x_0  \ge 0$ (resp. $x_0  \le 0$).
Then, our construction starts with $BES^\delta(x_0)$
(resp. $- BES^\delta(- x_0)$) which is clearly non-negative
(resp. non-positive). The constructed solution
is again non-negative (resp. non-positive) since
it is supported by an equivalent probability measure.
\end{preuve}
\begin{remark} \label{RNonUniq}
As we have mentioned in Proposition \ref{nu} and its extension to  $\delta =1$,
the martingale problem in the sense of Definition \ref{D31} admits
no uniqueness in general, at least with respect to $\shd_{L^\delta}$,
i.e. on the whole line.

\end{remark}
\subsection{Some preliminary results on a path-dependent SDE}

$ $

Before studying a new class of path-dependent martingale problems
we recall some results stated in Section 4.5 of \cite{ORT1_PartI}.

Let $\sigma_0: \R \rightarrow \R$.
Let $\bar \Gamma: \Lambda \rightarrow \R$ be a generic Borel functional.
Related to it
we formulate the following, which was Assumption 4.25 in \cite{ORT1_PartI}.
\begin{assumption} \label{ALip}
	$ $
	\begin{enumerate}
		\item There exists a function $l: \R_+ \rightarrow \R_+$ such that $\int_{0}^{\epsilon}l^{-2}(u)du = \infty$ for all $\epsilon > 0$ and $$|\sigma_{0}(x) - \sigma_{0}(y)| \leq l(|x - y|).$$
		\item
		$\sigma_0$ has at most linear  growth.
		\item  There exists $K > 0$ such that
		$$ |\bar{\Gamma}(s, \eta^1) - \bar{\Gamma}(s, \eta^2) \vert
		\le K\left(\vert\eta^1(s) - \eta^2(s)  \vert + \int_0^s
		\vert \eta^1(r) - \eta^2(r)\vert dr\right), $$
		for all $s \in [0,T], \eta^1, \eta^2 \in C([0,T]).$
		\item $ \bar \Gamma_\infty:= \displaystyle\sup_{s \in [0, T]} \vert \bar \Gamma(s, 0)\vert < \infty.$
	\end{enumerate}
\end{assumption}
The proposition below was the object of
	\cite[Proposition 4.27]{ORT1_PartI}.

\begin{proposition} \label{TStrong}
	Let $y_0 \in \R$.
	Suppose the validity
	of
	Assumption \ref{ALip}.
	Then $E(\sigma_0,0,\bar \Gamma)$, i.e.
	\begin{equation}\label{Ybis1}
	Y_t = y_0 + \int_{0}^{t}\sigma_0(Y_s)dW_s + \int_{0}^{t} \bar \Gamma(s, Y^s)ds,
	\end{equation}
	admits pathwise uniqueness.
\end{proposition}

The lemma below was the object of \cite[Lemma 4.28]{ORT1_PartI}.
\begin{lemma} \label{L1}
	Suppose the validity of the assumptions of Proposition \ref{TStrong}.
	Let $Y$ be a  solution of
	\eqref{Ybis1} and $m \geq 2$ an integer.
	Then there exists a constant $C>0$, depending
	on the linear growth constant of $\sigma_0$,
	$Y_0$, $K, T, m$ and the quantity (3)-(4) in Assumption \ref{ALip}
	such that
	$$ \E\left(\sup_{t \le T} \vert Y_s\vert^m\right) \le C.$$
\end{lemma}

\subsection{A new class of solutions to the martingale problem}

$ $

Besides Proposition \ref{P61},
 Proposition \ref{exi} below and Proposition \ref{X-S} provide
a new class of solutions to the martingale problem related
to $\shl$ with respect to $\shd_{L^\delta}$.
We consider now a particular case of $\bar \Gamma$, which
is associated with $\Gamma$:
 \begin{equation}\label{barra}
\bar{\Gamma}(s, \eta) := 2\sqrt{\vert \eta(s) \vert}\Gamma(s,
\sqrt{\vert \eta^s \vert})
+ \delta, \ s \in [0, T] \ \eta \in C([0, T]).
 \end{equation}
Next, we introduce a growth assumption on $\Gamma$.
\begin{assumption} \label{AGrowth}
 $\Gamma$ is continuous
and there exists a constant $K$ such that, for every $(s, \eta) \in \Lambda$
we have
$$  |\Gamma(s, \eta) \vert
 		\le K \left(1 +  \sup_{r \in [0,T]}\sqrt{|\eta(r)|}\right).$$
\end{assumption}
\begin{proposition}\label{exi}.
  Let $\delta \in [0,1]$.
Suppose that $\Gamma$ fulfills Assumption \ref{AGrowth}.
 Then, we have the following.
\begin{enumerate}
\item  Let $s_0 \ge 0$. The path-dependent SDE
\begin{equation}\label{S1bis}
	S_t = s_0 + \delta t + \int_{0}^{t}2\sqrt{|S_s|}dW_s + \int_{0}^{t}2\sqrt{|S_s|}\Gamma\left(s, \sqrt{|S^s|}\right)ds, \delta \geq 0,
 \end{equation}
admits existence in law, see Definition A.4 of Appendix in \cite{ORT1_PartI}.
\item The constructed solution
of \eqref{S1bis} in item (1) is non-negative.

\item Let $x_0 \geq 0$.
 The martingale problem related to
 $\shl f = L^\delta f + \Gamma f'$
(see Definition \ref{D31}, \eqref{DOperatorL}) admits existence   with respect to $\shd_{L^\delta}(\R_+)$.
\end{enumerate}

\end{proposition}	
\begin{preuve} \

  We remark that the hypothesis on $\Gamma$ implies
 that $\bar \Gamma$  has linear growth, i.e.
there is a constant $K$  such that
\begin{equation} \label{ELinGrowth}
 \bar \Gamma(t,\eta^t) \le K(1 + \sup_{s\in [0,t]} \vert \eta(s) \vert),
\forall (t,\eta) \in [0,T]\times C([0,T]).
\end{equation}
For item (1), we start truncating $\Gamma$.
  Let $N > 0$. Let us define, for $s \in [0, T], \eta \in C([0, T])$,
\begin{eqnarray*}
\Gamma^N(s,\eta)&:=& (\Gamma(s,\eta^s) \vee (- N)) \wedge N, \\
\bar \Gamma^N(s,\eta) &:=& 2 \sqrt{\vert \eta(s) \vert}
 \Gamma^N(s,\sqrt{\vert \eta \vert})
                   + \delta.
                     \end{eqnarray*}

                     \noindent We consider the SDE
\begin{equation}\label{aproxS}
\left\{
\begin{array}{rl}
dS_t= & \hbox{} \ 2\sqrt{|S_t|}dW_t + \bar{\Gamma}^N\left(t, S\right)dt,\\
S_0=& \hbox{} \ s_0.\\
\end{array}
\right.
\end{equation}
We set $x_0:= \sqrt{s_0}$.
Since $\Gamma^N$ is bounded, by
 Proposition \ref{P61}, the martingale problem related to $\shl$
  with respect to
$\shd_{L^\delta}$, admits a solution $(X,\P)$
which is non-negative.
By Proposition \ref{X-S} the SDE \eqref{aproxS}
admits existence in law and in particular there exists a solution
$S^N$ (which is necessarily non-negative)
on some probability space $(\Omega, \mathcal{F}, \bar \P^N)$.
     By It\^o's formula, this implies that (on the mentioned space),
      	\begin{equation} \label{EMartN}
	M^N_t := f(S^N_t) - f(S^N_0) - \int_{0}^{t} f'(S^N_s)
\bar{\Gamma}^N \left(s, S^N\right)  ds - 2 \int_{0}^{t}f''(S^N_s) |S^N_s|ds,
	\end{equation}
	is a martingale for all $f \in C^2$ with compact support.
        This will be used later.
	
   \noindent      We want first to show that the family of laws $(\bar{\Q}^N)$
of $(S^N)$   is tight. For this we are going to use the Kolmogorov-Centsov Theorem.
   We denote by $\bar \E^N$ the expectation related to $\bar{\P}^N$.
 According to  Problem 4.11 in Section 2.4 of \cite{ks},
 it is enough to find constants $\alpha, \beta > 0$ realizing
 \begin{equation}\label{KolCen}
\sup_{N}\bar \E^N (|S^N_t - S^N_s|^{\alpha}) \leq c|t - s|^{1 + \beta}; s, t \in [0, T],
\end{equation}
for some constant $c > 0$.
Indeed, we will  show \eqref{KolCen} for $\alpha = 6$ and $ \beta = 1$.
By \eqref{aproxS} and  Burkholder-Davis-Gundy inequality there exists a
constant $c_6$ such that,
for $0 \le s \le t  \le T$,
\begin{equation} \label{EBDG}
\bar {\E}^N (|S^N_t - S^N_s|^6) \leq c_6\left(\bar {\E}^N\left(\int_{s}^{t}(|S^N_r|)dr\right)^3
 + \bar \E^N \left(\int_{s}^{t} \bar{\Gamma}^N \left(r, |S^N|\right)dr\right)^6\right).
\end{equation}
By \eqref{ELinGrowth}, there exists a constant $\shc_1$
where
\begin{equation} \label{EBDG1}
|\bar{\Gamma}^N(s, \eta)| \leq 2\sqrt{|\eta(s)|}|\Gamma(s, \sqrt{|\eta|})| + \delta = |\bar{\Gamma}(s,\eta)| \leq \shc_1\left(1 + \sup_{r \leq s}|\eta(s)|\right),
\end{equation}
for every $(s, \eta) \in \Lambda$, uniformly in $N$.
\noindent
By Jensen's inequality and \eqref{EBDG1},
 there exists a constant $\shc_2 > 0$, only depending on $T$
and on
 $\bar{\Gamma}$, but not on $N$,
such that
$$ \bar{\E}^N(|S^N_t - S^N_s|^6) \leq \shc_2 \left((t - s)^2
\bar{\E}^N \left(\sup_{s \leq t}|S^N_s|^3\right) +
(t - s)^5 \bar{\E}^N \left(\sup_{s \leq t}|S^N_s|^6   \right)\right).$$

\noindent
By Lemma \ref{L1}, the quantity
$$  \bar{\E}^N \left(\sup_{s \leq T}|S^N_s|^3 + \sup_{s \leq T}|S^N_s|^6\right), $$
is bounded uniformly in $N$
and therefore \eqref{KolCen} holds.
Consequently, the family of laws $(\bar{\Q}^N)$
of $(S^N)$ under $(\bar{\P}^N)$ is tight.
We can therefore extract a subsequence which, for simplicity, we will still
call $\bar{\Q}^N$ that converges weakly to a probability measure $\bar{\Q}$
on $(C[0, T], \mathfrak{B}(C[0, T]))$.

\noindent
We denote by $ \E^N$ the expectation with respect to $\bar \Q^N$ .
Let $0 \le s \le t \le T$ and let  $F:C([0, s]) \rightarrow \R$ be a bounded and continuous function.
By \eqref{EMartN}, if $S$ is the canonical process we have
\begin{equation}\label{marcharBis}
\E^N((\tilde M^N_t - \tilde M^{N}_s)F(S_r , 0 \leq r \leq s)) = 0,
\end{equation}
where
\begin{equation} \label{EMartNBis}
	\tilde M^N_t := f(S_t) - f(S_0) - \int_{0}^{t} f'(S_s)
\bar{\Gamma}^N \left(s, S\right)  ds - 2 \int_{0}^{t}f''(S_s) |S_s|ds.
	\end{equation}
\noindent
By  Skorokhod's convergence theorem, there exists a sequence of
processes $(Y^N)$ and a process $Y$
both on a probability space $(\Omega, \mathcal{F}, \Q),$
converging u.c.p. to $Y$ as $N \rightarrow +\infty$.
Indeed $(Y^N)$ and $Y$ can be seen as
random elements taking values in the
state space $(C[0, T], \mathfrak{B}(C[0, T]))$.

Moreover, the law of $Y^N$ is $\bar{\Q}^N$,
 so that
\begin{equation}\label{marcharTer}
\E^\Q((\overline M^N_t - \overline M^{N}_s)F(Y^N_r,  0 \leq r \leq s)) = 0,
\end{equation}
where
\begin{equation} \label{EMartNTer}
	\overline M^N_t := f(Y^N_t) - f(S_0) - \int_{0}^{t} f'(Y^N_r)
\bar{\Gamma}^N \left(s, Y^N\right)  ds - 2 \int_{0}^{t}f''(Y^N_s) |Y^N_r|dr.
	\end{equation}
           \noindent    We wish to pass to the limit when $N \rightarrow \infty$
using Lebesgue dominated convergence theorem
        and obtain
\begin{equation}\label{marcharQuater}
\E^\Q((\overline M_t - \overline M_s)F(Y_r,  0 \leq r \leq s)) = 0,
\end{equation}
with
\begin{equation} \label{EMartNQuater}
	\overline M_t := f(Y_t) - f(S_0) - \int_{0}^{t} f'(Y_s)
\bar{\Gamma} \left(s, Y\right)  ds - 2 \int_{0}^{t}f''(Y_r) |Y_r|dr.
	\end{equation}
For this        it remains to prove that, when $N \rightarrow \infty$
\begin{equation} \label{EELimit}
        \E^\Q \left(\int_s^t f'(Y^N_r) \bar \Gamma^N(r, Y^N) dr \right)
\rightarrow \E^\Q \left(\int_s^t f'(Y) \bar \Gamma(r, Y) dr \right)
\end{equation}
and
\begin{equation} \label{EELimitY}
        \E^\Q \left(\int_s^t f''(Y^N_r) \vert Y^N_r \vert dr \right)
\rightarrow \E^\Q \left(\int_s^t f''(Y_r) \vert Y_r \vert dr \right),
\end{equation}
\noindent
as $N \rightarrow \infty$.
Below, we only prove
\eqref{EELimit} since \eqref{EELimitY} follows similarly.


Note that \eqref{EELimit} is true, if and only if,
 $$\displaystyle\lim_{N\to \infty} I_1(N)= 0, \quad  \displaystyle\lim_{N\to \infty} I_2(N)  = 0,$$
where \begin{align*}
I_1(N) := \E^{\Q}\left[\int_{s}^{t}f'(Y^N_r)(\bar{\Gamma}^N(r, Y^N) - \bar{\Gamma}(r, Y^N))dr\right], \\
I_2(N) := \E^{\Q}\left[\int_{s}^{t}f'(Y^N_r)\bar{\Gamma}(r, Y^N) - f'(Y_r)\bar{\Gamma}(r, Y) dr\right].
\end{align*}
\noindent
By \eqref{ELinGrowth}
and \eqref{EBDG1}, we have \begin{align*}
I_1(N) \leq ||f'||_{\infty}\E^{\Q}\left[1_{\{\sup_{r \in [0, T]}|\Gamma(r, Y^{N, r})| > N\}}\int_{s}^{t}|\bar{\Gamma}^N(r, Y^N) - \bar{\Gamma}(r, Y^N)|dr\right] \leq \\
 \leq 2 K T||f'||_{\infty}\E^{\Q}\left[1_{\{\sup_{r \in [0, T]}|\Gamma(r, Y^{N})| > N\}}
(1 + \sup_{r \in [0,T]}|Y^N_r|)\right].
\end{align*}
By Cauchy-Schwarz's inequality, there exists a non-negative constant $C(f, T, K)$ such that
\begin{equation} \label{E112N}
I_1(N)^2 \leq C(f, T, K)I_{11}(N)I_{12}(N),
\end{equation}
where \begin{align*}
I_{11}(N) := \Q\left(\sup_{r \in [0, T]}|\Gamma(r, Y^{N})| > N\right),\\
I_{12}(N) := \E^{\Q}\left[1 + \sup_{r \in [0,T]}|Y^N_r|^2\right].
\end{align*}
\noindent
By Chebyshev's inequality we have \begin{align*}
I_{11}(N) \leq \dfrac{1}{N^2} \E^{\Q}\left[\sup_{r \in [0,T]} |\Gamma(r, Y^{N})|^2 \right] \leq \dfrac{2 K}{N^2}\E^{\Q}\left[1 + \sup_{r \in [0,T]} |Y^N_r|^2\right].
\end{align*}
Consequently, $\displaystyle\lim_{N\to \infty}I_{11}(N) = 0$ because of Lemma \ref{L1}. On the other hand, again by Lemma \ref{L1},
$I_{12}(N)$ is bounded in $N$ and so by \eqref{E112N}, we get
 $\displaystyle\lim_{N\to \infty}I_{1}(N) = 0$.

Concerning $I_2(N),$ we have
\begin{equation} \label{E2N}
I_2(N)^2 \leq T \int_{s}^{t}\E^{\Q}\left[|f'(Y^N_r)\bar{\Gamma}(r, Y^N) - f'(Y_r)\bar{\Gamma}(r, Y)|^2\right]dr.
\end{equation}
 By Lemma \ref{L1}, there exists a constant $C$ not depending on
 $N$ such that
$$\E^{\Q}\left[\sup_{r \in [0,T]}|Y^N_r|^4\right] \leq C,$$
and, consequently, by Fatou's Lemma
$$\E^{\Q}\left[\sup_{r \in [0,T]}|Y_r|^4\right] \leq C.$$
Let $r \in [0,T]$. We have
\begin{eqnarray}\label{E4b}
 \E^{\Q}[|f'(Y^N_r)\bar{\Gamma}(r, Y^N) &-& f'(Y_r)\bar{\Gamma}(r, Y)|^4]
    \\
&\le& 8 ||f'||^4_{\infty} K^4\left(2 + \E^{\Q}\left[\sup_{r \in [0,T]}|Y^N_r|^4 +
 \sup_{r \in [0,T]}|Y_r|^4\right]\right) \nonumber \\
&\leq& 16 ||f'||^4_{\infty}K^4 (1 + C). \nonumber
\end{eqnarray}
So the sequence
$$|f'(Y^N_r)\bar{\Gamma}(r, Y^N) - f'(Y_r)\bar{\Gamma}(r, Y)|^2$$
is uniformly integrable.
We fix again $r \in [0,T].$ Since $f'$ and $\bar \Gamma$ are continuous
it follows that
\begin{equation} \label{Eentire}
\E^{\Q}\left[|f'(Y^N_r)\bar{\Gamma}(r, Y^N) - f'(Y_r)\bar{\Gamma}(r, Y)|^2\right]\longrightarrow 0,
\end{equation}
as $N \rightarrow \infty$.
Now \eqref{E4b} and Cauchy-Schwarz implies that
\begin{equation}\label{E4c}
\E^{\Q}[|f'(Y^N_r)\bar{\Gamma}(r, Y^N) - f'(Y_r)\bar{\Gamma}(r, Y)|^2]
\leq 4 ||f'||^2_{\infty} K^2 \sqrt{1 + C}.
\end{equation}
This time \eqref{Eentire}, \eqref{E4c} and
Lebesgue's dominated theorem show
that the entire Lebesgue integral of \eqref{Eentire} on $[s, t]$
converges to $0$.
Finally, $\displaystyle\lim_{N\to \infty}I_2(N) = 0$
so that  we conclude to \eqref{EELimit}
 and, consequently, \eqref{marcharQuater}.
Therefore, $(Y,\Q)$ solve the martingale problem
of the type \eqref{Ef32} as in Proposition \ref{P31}  with
$$ Lf(x) = 2 \vert x \vert f''(x) + \delta f'(x)$$
and $\bar \Gamma$ replacing $\Gamma$.
By Proposition \ref{P31}, this concludes the proof of item (1).

\medskip
Concerning item (2), the previously constructed $Y$
 is a (weak) solution to \eqref{S1bis}
under the probability $\Q$. Since it is a limit
of non-negative solutions, it will also be non-negative.

\medskip

Item (3) follows from Proposition \ref{X-S} below.

\end{preuve}

\subsection{Equivalence between martingale problem and SDE in the
  path-dependent case}

\label{SEquivalence}

$ $

We state here an important result establishing the equivalence
between the martingale problem and a path-dependent SDE of
squared Bessel type. Let $0 \le \delta \le 1$.

\begin{proposition}\label{X-S} Let $(\Omega, \shf, \P)$
be a probability space.
 Let $X$ be a
 stochastic process and we denote $S = X^2$.
\begin{enumerate}
\item
 $(|X|, \P)$ is a solution to the
 martingale problem related to \eqref{DOperatorL}
 with respect to
$\shd_{L^{\delta}}$,
 if and only if, the process $S$ is a solution of \eqref{S1bis}
for some $\mathfrak{F}^X$-Brownian motion $W$.
\item Let $W$ be a standard Brownian motion (with respect to $\P$). Then
  $|X|$ is a solution
to the strong
 martingale problem with respect
to  $\shd_{L^{\delta}}$ and
$W$, if and only if, $S$ is a solution of \eqref{S1bis}.
\end{enumerate}
\end{proposition}
\begin{remark} \label{RX-S}
In the  statement of Proposition \ref{X-S},
$\shd_{L^{\delta}}$ can be replaced
with $\shd_{L^{\delta}}(\R_+)$, provided that $\vert X \vert$ is replaced
by $X$.

\end{remark}

\begin{prooff} (of Proposition \ref{X-S}).
	$ $
We discuss item (1).

Concerning the direct implication, by choosing $f_1(x) = x^2, f_2(x) = x^4$ we have $L^{\delta}f_1(x) = \delta, L^{\delta}f_2(x) = 2(2 + \delta)x^2$. By definition of the martingale problem,
the two processes ($ t \in [0,T]$)
\begin{equation}\label{e1a} M_t := X^2_t - X^2_0 - \delta t - \int_{0}^{t}2|X_s|\Gamma(s,|X^s|)ds\end{equation}
and  \begin{equation}\label{e2a} N_t := X^4_t - X^4_0 - 2(2 + \delta)\int_{0}^{t}X^2_s ds - 4\int_{0}^{t}|X_s|^3\Gamma(s, |X^s|)ds,\end{equation}
are $\mathfrak{F}^X$-local martingales.

Since $S = X^2$, by \eqref{e1a} we have $[S] = [M]$.
By integration by parts and \eqref{e1a}, we have
$$
[M]_t = [X^2]_t = X^4_t - X^4_0 - 2 \int_{0}^{t}X^2_s dX^2_s = X^4_t - X^4_0 - 2\delta\int_{0}^{t}X^2_s ds - 4\int_{0}^{t}|X_s|^3\Gamma(s, |X^s|)ds + M^1,
$$
where $M^1$ is a local martingale.  Therefore
\begin{equation}\label{e1tera}
 X^4_t - X^4_0 = M^1 +  2\delta\int_{0}^{t}X^2_s ds + 4\int_{0}^{t}|X_s|^3\Gamma(s, |X^s|)ds + [M]_t, \quad t \in [0,T].
\end{equation}
\eqref{e1tera} and \eqref{e2a}
 give us two decompositions of the semimartingale $X^4$;
 by the uniqueness of the semimartingale decomposition, $[M]_t = 4\int_{0}^{t}X^2_s ds$.
We set
  \begin{equation}\label{e3a}
 W_t = \int_{0}^{t} \frac{dM_s}{2|X_s|}, t \in [0,T].
 \end{equation}
 By L\'evy's characterization theorem,
 $W$ is an $\mathfrak{F}^X$-Brownian motion and by \eqref{e1a},
 we conclude that $S_t = s_0 + \delta t + \int_{0}^{t}2\sqrt{S_s}dW_s + \int_{0}^{t}2\sqrt{S_s}\Gamma(s, \sqrt{S^s})ds, t \in [0,T].$

Concerning the converse implication, suppose that $S = X^2$ solves \eqref{S1bis}
for some Brownian motion $W$. Then $S$ solves
\begin{equation}\label{S1ter}
	S_t = s_0 + \delta t + \int_{0}^{t}2\sqrt{|S_s|}d\widetilde W_s, t \in [0,T],
 \end{equation}
where
$$ \widetilde W_t := W_t + \int_{0}^{t} \Gamma(s, \sqrt{|S^s|})ds, t \in [0,T].$$
Let $f \in \shd_{L^\delta}$; by Proposition \ref{S} and Remark \ref{RS}
we have
\begin{equation}\label{lmp2quater}
 f(|X_t|) - f(|X_0|) - \int_{0}^{t}L^{\delta}f(|X_s|)ds =
\int_{0}^{t}f'(|X_s|) d\widetilde W_s.
 \end{equation}
Consequently
\begin{eqnarray*}
M^f_t &:=& f(|X_t|) - f(|x_0|) - \int_{0}^{t}\mathit{L}^\delta f(|X_s|)ds - \int_{0}^{t}f'(|X_s|)\Gamma(s, |X^s|)ds \\
 &=& \int_0^t f'(|X_s|)  dW_s,
\end{eqnarray*}
is an $\mathfrak{F}^X$-local martingale.
Then, $(\vert X \vert, \P)$ solve the martingale problem
related to \eqref{DOperatorL}
  with respect to
 $\shd_{L^\delta}$ in the sense of Definition \ref{D31}.
On the other hand, $\vert X\vert$ also solves
the strong martingale problem with respect to $\shd_{L^\delta}$ and $W$.
This concludes the proof of item (1).

As far as item (2) is concerned, the converse implication
argument can be easily adapted to the argument for the proof of the converse implication
in (1).
Concerning the direct implication, we define $f_1$
as in the proof of  item (1). By \eqref{lmpBis}, \eqref{e1a}
 and the fact that
$$ M = 2 \int_0^\cdot f_1'(\vert X_s\vert) dW_s =
 2 \int_0^\cdot \sqrt{S_s} dW_s,$$
we  obtain \eqref{S1bis}. This concludes the proof.

\end{prooff}

\subsection{The martingale problem in the path-dependent
case: uniqueness in law.}

$ $

A consequence of Girsanov's theorem gives us the following.

\begin{proposition} \label{P64} Let $0 \leq \delta \leq 1$.
	Suppose that $\Gamma$ is bounded. The martingale problem
related to \eqref{DOperatorL}
        with respect to $\shd_{L^{\delta}}(\mathbb{R_+})$
admits uniqueness.
\end{proposition}
\begin{remark} \label{R64} Let $x_0 \ge 0$ (resp. $x_0 \le 0$).
  By Proposition \ref{P61}, every solution
of the aforementioned martingale problem
 is non-negative (resp. non-positive).
\end{remark}

\begin{prooff} (of Proposition \ref{P64}).

Let $(X^i, \P^i), i =1,2$ be two solutions to the martingale problem related to $\shl f = Lf + \Gamma f'$
with respect  to $\shd_{L^{\delta}}(\R_+)$. By Proposition \ref{X-S}, $S^i = (X^i)^2$ is a solution of \eqref{S1bis}, for some Brownian motion $W^i$ and $\P^i$.
We define the random variable (which is also a Borel functional
of $X^i$)
\begin{equation*}
V^{i}_t
 = \exp\left(-\int_{0}^{t} \Gamma(s, X^{i}) dW^{i}_s - \frac{1}{2}\int_{0}^{t}\left(\Gamma(s, X^{i})\right)^{2}d s\right).
\end{equation*}
By the Novikov condition, it is a $\P^i$-martingale. This allows us to
define the probability $d\Q^i = V_T^i d\P^i$.
By Girsanov's theorem, for $i = 1, 2$,
 under
$\Q^{i},$  $B^i_t := W^{i}_t + \int_{0}^{t} \Gamma(s, X^{i,s})ds$
is a Brownian motion.
Therefore,
$S^i$
 is a
solution of \eqref{S1bis}  with $\Gamma = 0$, under $\Q^i.$
Now \eqref{S1bis} (with $\Gamma = 0$) admits pathwise uniqueness
and therefore uniqueness in law, by Yamada-Watanabe theorem.
Consequently $S^{i}$ (under $\Q^{i}$), $i=1,2$ have the same law
and the same holds of course for
$X^{i}, i=1,2$.
Hence, for every Borel set $B \in \mathfrak{B}(C[0, T])$ we have
\begin{equation*}
\P^{1}\{X^{1} \in B\} \\=\\ \int_{\Omega}\dfrac{1}{V^1_{T}(X^{1})}1_{\{X^1 \in B\}}d\Q^{1}\\=\\
\int_{\Omega}\dfrac{1}{V^2_{T}(X^{2})}1_{\{X^2 \in B\}}d\Q^{2}\\=\\
\P^{2}\{X^2 \in B\}.
\end{equation*}
So, $X^1$ under  $\P^1$ has the same law as $X^2$ under $\P^2$.
Finally the martingale problem related to \eqref{DOperatorL} with respect
to $\shd_{L^{\delta}} (\R_+)$ admits uniqueness.
\end{prooff}

\subsection{Path-dependent Bessel process: results on
pathwise uniqueness.}\label{SPathUniq}

$ $

In this section, $\bar \Gamma$ is the same as the one defined in \eqref{barra}, i.e.
\begin{equation*}
  \bar{\Gamma}(s, \eta) := 2\sqrt{\vert \eta(s)\vert}
  \Gamma(s, \sqrt{\vert \eta^s\vert }) + \delta, s \in [0, T]\ \eta \in C([0, T]).
\end{equation*}
At this point, we can state a pathwise uniqueness
theorem.
For this purpose, we state the following assumption.
\begin{assumption}\label{H3}
		$ $
	\begin{enumerate}
        \item There exists a constant $K > 0$ such that,
        for every $s \in [0,T], \eta^1, \eta^2 \in C([0,T])$, we have
          $|\bar{\Gamma}(s, \eta^1) - \bar{\Gamma}(s, \eta^2) \vert
		\le K\left(\vert\eta^1(s) - \eta^2(s)  \vert + \int_{0}^{s}|\eta^1(r) - \eta^2(r)|dr\right).$
		\item $\displaystyle\sup_{t \in [0, T]} \vert \bar \Gamma(t, 0)\vert < \infty.$
	\end{enumerate}
\end{assumption}
\begin{remark}\label{r}
	
	$ $
\begin{enumerate}

\item
 $\sigma_{0}(y) = 2\sqrt{\vert y\vert}$  has linear growth.
\item	
  Defining $l(x) = 2\sqrt{x}, x \ge 0$,
  we have
  $\int_{0}^{\epsilon} l^{-2}(u) du = \infty$
for every  $\epsilon > 0$ and $|l(x) -l(y)| \leq l(|x - y|), x, y \in \R_+$.

\end{enumerate}
\end{remark}

\begin{remark}\label{r2}
	Note that, by Remark \ref{r},
 Assumption \ref{H3}
implies Assumption \ref{ALip}.
\end{remark}
We start the analysis by considering equation \eqref{S1bis}. For the definitions of strong existence and pathwise uniqueness for path-dependent SDEs,
see Definitions A.2 and A.3 of \cite{ORT1_PartI}.
\begin{theorem}\label{uni2} Suppose Assumptions \ref{H3} and \ref{AGrowth}.
\begin{enumerate}
\item   \eqref{S1bis} admits pathwise uniqueness.
\item
 \eqref{S1bis} admits strong existence.
\item Suppose $x_0 \ge 0$.
Every  solution of \eqref{S1bis} with $s_0 = x_0^2$ is non-negative.
\end{enumerate}
\end{theorem}

\begin{prooff}  \

\begin{enumerate}
\item
We remark that \eqref{S1bis} is of the form \eqref{Ybis1}.
The result follows  from Proposition \ref{TStrong}
and Remark \ref{r}.
\item
By Proposition \ref{exi}, we have existence in law.
By  an extension of
Yamada-Watanabe theorem
to the path-dependent case,
strong existence holds for  \eqref{S1bis}.
\item Suppose $x_0 \ge 0$.
 By Proposition \ref{exi} (2),
\eqref{S1bis}
admits even existence in law of  a non-negative solution.
By  Yamada-Watanabe theorem extended to the path-dependent case,
pathwise uniqueness implies uniqueness in law, so that the
above-mentioned solution has to be non-negative.
\end{enumerate}
\end{prooff}

We are now able to state the following.
\begin{corollary}\label{Solution}
	Suppose that
$\bar{\Gamma}$ (defined in $\eqref{barra}$) fulfills Assumptions \ref{H3}
 and \ref{AGrowth}.
Then the strong martingale problem related to \eqref{DOperatorL}
(see Definition \ref{DSolution}) with respect to $\shd_{L^{\delta}}(\R_+)$
and $W$ admits strong existence and pathwise uniqueness.
\end{corollary}
\begin{preuve} \
By Theorem \ref{uni2},  the equation \eqref{S1bis} admits
a unique strong solution which is non-negative.
 Proposition \ref{X-S} and Remark \ref{RX-S} allow us to conclude the proof.

\end{preuve}

{\bf ACKNOWLEDGEMENTS.}
The authors are grateful to the anonymous Refere for stimulating
remarks and questions which motivated them to improve the paper.
The research related to this paper
 was financially supported by the Regional Program MATH-AmSud 2018,
 project Stochastic analysis of non-Markovian phenomena (NMARKOVSOC),
 grant 88887.197425/2018-00. A.O. acknowledges the financial support of
 CNPq Bolsa de Produtividade de Pesquisa grant 303443/2018-9.

 .

\bibliographystyle{plain}
\bibliography{../../../../../../BIBLIO_FILE/biblio-PhD-Alan}
\end{document}